\documentclass[10pt]{amsart}
\usepackage[all,cmtip]{xy}
\usepackage[english]{babel}
\usepackage[latin2jk]{inputenc}
\usepackage{amssymb}
\usepackage{bbm}
\usepackage{amsthm}
\usepackage{amscd}

\def\today{\number\day\space\ifcase\month\or   January\or February\or
   March\or April\or May\or June\or   July\or August\or September\or
   October\or November\or December\fi\   \number\year}

\theoremstyle{definition}
\newtheorem{thm}{Theorem}[section]
\newtheorem{lemma}[thm]{Lemma}
\newtheorem{prop}[thm]{Proposition}
\newtheorem{df}[thm]{Definition}
\newtheorem{cor}[thm]{Corollary}

\newtheorem{rem}[thm]{Remark}
\newtheorem{rems}[thm]{Remarks}
\newtheorem{nota}[thm]{Notation}

\newtheorem{eg}[thm]{Example}
\newtheorem{egs}[thm]{Examples}
\newtheorem{pbm}[thm]{Problem}

\newcommand{\beq}{\begin{equation}}
\newcommand{\eeq}{\end{equation}}
\newcommand{\beqa}{\begin{eqnarray*}}
\newcommand{\eeqa}{\end{eqnarray*}}
\newcommand{\bal}{\begin{align*}}
\newcommand{\eal}{\end{align*}}
\newcommand{\bi}{\begin{itemize}}
\newcommand{\ei}{\end{itemize}}
\newcommand{\be}{\begin{enumerate}}
\newcommand{\ee}{\end{enumerate}}

\newcommand{\dt}{\delta}
\newcommand{\ep}{\varepsilon}
\newcommand{\zt}{\zeta}

\newcommand{\Q}{{\mathbb{Q}}}

\newcommand{\Z}{{\mathbb{Z}}}
\newcommand{\R}{{\mathbb{R}}}
\newcommand{\C}{{\mathbb{C}}}
\newcommand{\N}{{\mathbb{N}}}
\newcommand{\NI}{{\overline{\mathbb{N}}}}

\newcommand{\B}{{\mathcal{B}}}

\newcommand{\T}{{\mathbb{T}}}

\newcommand{\dist}{{\mathrm{dist}}}

\newcommand{\spec}{{\mathrm{sp}}}

\newcommand{\Aut}{{\mathrm{Aut}}}
\newcommand{\Max}{{\mathrm{Max}}}

\newcommand{\I}{\infty}

\title[]{Banach algebras generated by an invertible isometry of an $L^p$-space}
\date{\today}

\author[Eusebio Gardella]{Eusebio Gardella}
\address{Eusebio Gardella
Department of Mathematics, Deady Hall, University of Oregon
Eugene OR, 97403, USA;
and Fields Institute for Research in Mathematical Sciences,
222 College Street, Toronto ON, M5T 3J1, Canada.}
\email{gardella@uoregon.edu}
\urladdr{http://pages.uoregon.edu/gardella/}
\author{Hannes Thiel}
\address{Hannes Thiel
Mathematisches Institut, Fachbereich Mathematik und Informatik der
Universit\"at M\"unster, Einsteinstrasse 62, 48149 M\"unster, Germany;
and Fields Institute for Research in Mathematical Sciences,
222 College Street, Toronto ON, M5T 3J1, Canada.}
\email{hannes.thiel@uni-muenster.de}
\urladdr{www.math.ku.dk/~thiel/}
\thanks{This work was completed while
the authors were attending the Thematic Program on Abstract Harmonic
Analysis, Banach and Operator Algebras at the Fields Institute in
January-June 2014. The hospitality of the Fields Institute is gratefully acknowledged.}

\begin{document}

\begin{abstract}
We provide a complete description of those Banach algebras that are generated
by an invertible isometry of an $L^p$-space together with its inverse.
Examples include the algebra $PF_p(\mathbb{Z})$ of
$p$-pseudofunctions on $\mathbb{Z}$, the commutative $C^*$-algebra $C(S^1)$
and all of its quotients, as well as uncountably many `exotic'
Banach algebras.

We associate to each isometry of an $L^p$-space, a spectral invariant called
`spectral configuration', which contains
considerably more information than its spectrum as an operator.
It is shown that the spectral configuration describes the isometric
isomorphism type of the Banach algebra that the isometry generates
together with its inverse.

It follows from our analysis that these algebras are semisimple.
With the exception of $PF_p(\mathbb{Z})$, they are all closed under
continuous functional calculus, and their Gelfand transform is an isomorphism.

As an application of our results, we show that Banach algebras that act on $L^1$-spaces
are not closed under quotients. This answers the case $p=1$ of a question asked
by Le Merdy 20 years ago.
\end{abstract}
\maketitle
\tableofcontents

\section{Introduction}

Associated to any commutative Banach algebra $A$ is its maximal ideal space $\Max(A)$, which is a locally
compact Hausdorff space when endowed with the hull-kernel topology. Gelfand proved in the 1940's that there is
a norm-decreasing homomorphism $\Gamma_A\colon A\to C_0(\Max(A))$, now called the Gelfand transform. This
representation of $A$ as an algebra of functions on a locally compact Hausdorff space is fundamental to any
study of commutative Banach algebras. It is well known that $\Gamma_A$ is injective if and only if $A$ is
semisimple, and that it is an isometric isomorphism if and only if $A$ is a $C^*$-algebra. The reader is referred
to \cite{kaniuth} for a extensive treatment of the theory of commutative Banach algebras.\\
\indent Despite the usefulness of the Gelfand transform, we are still far from understanding the isometric
structure of (unital, semisimple) commutative Banach algebras, since the Gelfand transform is almost never
isometric. Commutative Banach algebras for which their Gelfand transform is isometric are called uniform algebras, an
example of which is the disk algebra. \\
\indent In this paper, we study those Banach algebras that are generated by an invertible isometry of an $L^p$-space
together with its inverse, for $p\in [1,\infty)$. These are basic examples of what Phillips calls $L^p$-operator algebras
in \cite{phillips crossed products}, which are by definition Banach algebras that can be isometrically represented
as operators on some $L^p$-space. $L^p$-operator algebras constitute a large class of Banach algebras,
which contains all not-necessarily selfadjoint operator algebras (and in particular, all $C^*$-algebras), as well
as many other naturally ocurring examples of Banach algebras. A class worth mentioning is that of the (reduced) $L^p$-operator
group algebras, here denoted $F^p_\lambda(G)$, associated to a locally compact group $G$. These are introduced
in Section 8 of \cite{herz:HarmSynth} with the name $p$-pseudofunctions. (We warn the reader that these algebras
are most commonly denoted be $PF_p(G)$, for example in \cite{herz:HarmSynth} and \cite{neufang runde}.)
The notation $F^p_\lambda(G)$ first appeared in \cite{phillips crossed products}, following
conventions used in \cite{DJW}, and was chosen to match the already established notation in $C^*$-algebra theory.)
The Banach algebra $F^p_\lambda(G)$ is the Banach subalgebra of $\B(L^p(G))$ generated by the image of the integrated
form $L^1(G)\to \B(L^p(G))$ of the left regular representation of $G$ on $L^p(G)$. It is commutative if and only if the group
$G$ is commutative, and together with the universal group $L^p$-operator algebra $F^p(G)$, contains a great deal of information about the group.
For example, it is a result in \cite{GarThi_GpsLp} (also proved independently by Phillips), that a locally
compact group $G$ is amenable if and only if the canonical map
$F^p(G)\to F^p_\lambda(G)$ is an isometric isomorphism, and when $G$ is discrete, this is moreover equivalent
to $F^p(G)$ being amenable as a Banach algebra.

Of particular interest in our development are the algebras $F^p(\Z)$ for $p\in [1,\infty)$. When $p=2$,
this is a commutative $C^*$-algebra which is canonically identified with $C(S^1)$ under the Gelfand transform.
However, $F^p(\Z)$ is never a $C^*$-algebra when $p$ is not equal to 2, although it can be identified with a dense
subalgebra of $C(S^1)$. As one may expect, it turns out that the norm on $F^p(\Z)$ is particularly hard to compute.\\
\indent Another important example of a commutative $L^p$-operator algebra is the group $L^p$-operator algebra
$F^p(\Z_n)$ associated with the finite cyclic group of order $n$. For a fixed $n$, this is the subalgebra of the
algebra of $n$ by $n$ matrices with complex entries generated by the cyclic shift of the basis, and hence is
(algebraically) isomorphic to $\C^n$. It should come as no surprise that the norm on $\C^n$ inherited via this
identification is also very difficult to compute (except in the cases $p=2$, where the norm is simply the supremum
norm, and the case $p=1$, since the 1-norm of a matrix is easily calculated). \\
\indent Studying the group of symmetries is critical in the understanding of any given mathematical structure. In the
case of classical Banach spaces, this was started by Banach in his 1932 book. There is now a great deal of literature
concerning isometries of Banach spaces. See \cite{fleming jamison book}, just to mention one example. For $L^p$-spaces,
it was Banach who first described the structure of invertible isometries of $L^p([0,1])$ for $p\neq 2$, although a complete proof was
not available until Lamperti's 1958 paper \cite{lamperti}, where he generalized Banach's Theorem to $L^p(X,\mu)$ for
an arbitrary $\sigma$-finite measure space $(X,\mu)$ and $p\neq 2$. The same proof works for invertible isometries between different
$L^p$-spaces, yielding a structure theorem for isometric isomorphisms between any two of them. Roughly speaking, an isometric
isomorphism from $L^p(X,\mu)$ to $L^p(Y,\nu)$, for $p\neq 2$, is a combination of a multiplication operator by a measurable function
$h\colon Y\to S^1$, together with an invertible measurable transformation $T\colon X\to Y$ which preserves null-sets.
While this is slighly inaccurate for general $\sigma$-finite spaces, it is true under relatively mild assumptions. (In
the general case one has to replace $T\colon X\to Y$ with a Boolean homomorphism between their $\sigma$-algebras, going
in the opposite direction.) In the case $p=2$, the maps described above are also isometries, but they are not the only
ones, and we say little about these.\\
\indent Starting from Lamperti's result, we study the Banach algebra generated by an invertible isometry of an $L^p$-space
together with its inverse. It turns out that the multiplication operator and the measurable transformation of the space
have rather different contributions to the resulting Banach algebra (more precisely, to its norm). While the multiplication
operator gives rise to the supremum norm on its spectrum, the measurable transformation induces a somewhat more exotic
norm, which, interestingly enough, is very closely related to the norms on $F^p(\Z)$ and $F^p(\Z_n)$ for $n$ in $\N$.\\
\ \\
\indent This paper is organized as follows.
In the remainder of this section we introduce the necessary notation and
terminology.
In Section 2, we recall a number of results on group algebras acting on $L^p$-spaces from \cite{GarThi_GpsLp}
and \cite{GarThi_functoriality}, particularly about cyclic groups. We also prove some results that will be
needed in the later sections. \\
\indent In Section 3, we introduce the notion of spectral configuration; see
Definition~\ref{df: spectral conf}.
For $p$ in $[1,\I)$, we associate to each spectral configuration an $L^p$-operator algebra
which is generated by an invertible isometry together with its inverse in Definition
\ref{df: Fpsigma}.
We also show that there is a strong dichotomy with respect to the isomorphism type of
these algebras: they are isomorphic to either $F^p(\Z)$ or to the space of all continuous
functions on its maximal ideal space; see Theorem \ref{thm: Fpsigma}.
We point out that the isomorphism cannot in general be chosen to be isometric in the second case.
The saturation of a spectral configuration (Definition \ref{df:saturated}) is
introduced with the goal of showing that for $p\in [1,\I)\setminus\{2\}$, two
spectral configurations have canonically isometrically isomorphic associated $L^p$-operator algebras if and only if their saturations
are equal; see Corollary \ref{cor: isom classif of Fpsigma}. \\
\indent In Section 4, we discuss the structure of isometric isomorphisms between $L^p$-spaces for $p\neq 2$. In Theorem \ref{thm: Lamperti}
we use the lifting results from \cite{fremlin} to show that, under mild assumptions, every isometric isomorphism between $L^p$-spaces
is a combination of a multiplication operator and a bi-measurable isomorphism between the measure spaces that preserves null-sets.\\
\indent Section 5 contains our main results. Theorem \ref{thm: description} describes the isometric isomorphism type
of the Banach algebra $F^p(v,v^{-1})$ generated by an invertible isometry $v$ of an $L^p$-space together with its inverse, for $p\neq 2$. It turns
out that this description is very closely related to the dynamic properties of the measurable transformation of the space, and we
get very different outcomes depending on whether or not it has arbitrarily long strings (Definition \ref{df: long string}).
A special feature of algebras of the form $F^p(v,v^{-1})$ is that they are always simisimple, and, except in the case
when $F^p(v,v^{-1})\cong F^p(\Z)$, their Gelfand transform is always an isomorphism (although not necessarily isometric);
see Corollary \ref{cor: semisimple, GTonto}.
Additionally, we show that algebras of the form $F^p(v,v^{-1})$ are closed by functional calculus of a
fairly big class of functions, which includes all continuous functions on the spectrum of $v$ except when $F^p(v,v^{-1})\cong F^p(\Z)$,
in which case only bounded variation functional calculus is available.

Finally, in Section 6, we apply our results to answer the case $p=1$ of a question posed by Le Merdy 20 years ago
(Problem~3.8 in~\cite{LeMerdy}). We show that the class of Banach algebras that act on $L^1$-spaces is not closed
under quotients. In \cite{GarThi_QuotLpOpAlgs}, we use the main results of the present work to give a negative
answer to the remaining cases of Le Merdy's question.

Further applications of the results contained in this paper will appear in \cite{GarThi_Irrat}, where we study
Banach algebras generated by two invertible isometries $u$ and $v$ of an $L^p$-space, subject to the relation
$uv=e^{2\pi i \theta}vu$ for some $\theta\in\R\setminus\Q$.\\
\ \\
\indent Before describing the notation and conventions we use, we mention that, unlike in the case of $C^*$-algebras, the Banach
algebra generated by an invertible isometry of an $L^p$-space does not necessarily contain its inverse, even for $p=2$, as the
following example shows. For a function $f$, we denote by $m_f$ the operator of multiplication by $f$.

\begin{eg} Denote by $\mathbb{D}$ the open disk in $\C$, and consider the disk algebra
$$A(\mathbb{D})=\{f\in C(\overline{\mathbb{D}})\colon f|_{\mathbb{D}} \mbox{ is holomorphic}\}.$$
Then $A(\mathbb{D})$ is a Banach algebra when endowed with the supremum norm. Denote by $\mu$ the Lebesgue measure on $S^1$ and
define a homomorphism $\rho\colon A(\mathbb{D})\to \B(L^2(S^1,\mu))$ by $\rho(f)=m_{f|_{S^1}}$ for $f$ in $A(\mathbb{D})$. Then
$\rho$ is isometric by the Maximum Modulus Principle. \\
\indent Denote by $\iota\colon \mathbb{D}\to \C$ the canonical inclusion. Then $\iota$ generates $A(\mathbb{D})$ because every
holomorphic function on $\mathbb{D}$ is the uniform limit of polynomials. Moreover, $\rho(\iota)$ is an invertible isometry of
$L^2(S^1,\mu)$, but $\iota$ is clearly not invertible in $A(\mathbb{D})$. We conclude that $A(\mathbb{D})$ is an $L^2$-operator
algebra generated by an invertible isometry, but it does not contain its inverse.\end{eg}

We take $\N=\{1,2,\ldots\}$ and $\NI=\N\cup\{\I\}$. For $n$ in $\N$, we write $\ell^p_n$ in place of $\ell^p(\{0,\ldots,n-1\})$.
Also, for $n$ in $\N$, we denote $\omega_n=e^{\frac{2\pi i}{n}}$.\\
\indent For Banach spaces $X$ and $Y$, we denote by $\B(X,Y)$ the Banach space of all bounded linear operators $X\to Y$, and write
$\B(X)$ in place of $\B(X,X)$. If $A$ is a unital Banach algebra and $a\in A$, we denote its spectrum in $A$ by $\spec_A(a)$, or
just $\spec(a)$ if no confusion as to where the spectrum is being computed is likely to arise.\\
\indent If $(X,\mathcal{A},\mu)$ is a measure space and $Y$ is a measurable subset of $X$, we write $\mathcal{A}_Y$ for the
restricted $\sigma$-algebra
$$\mathcal{A}_Y=\{E\cap Y\colon E\in \mathcal{A}\},$$
and we write $\mu|_Y$ for the restriction of $\mu$ to $\mathcal{A}_Y$. Note that if $\{X_n\}_{n\in\N}$ is a partition of $X$
consisting of measurable subsets, then there is a canonical isometric isomorphism
$$L^p(X,\mu)\cong \bigoplus\limits_{n\in\N}L^p(X_n,\mu|_{X_n}).$$
(The direct sum on the right-hand side is the $p$-direct sum.)

We will usually not include the $\sigma$-algebras in our notation for measure spaces, except when they are necessary (particularly
in Section 4). The characteristic function of a measurable set $E$ will be denoted $\mathbbm{1}_E$.

For $p\in (1,\I)$, we denote by $p'$ is conjugate (H\"older) exponent, which satisfies $\frac{1}{p}+\frac{1}{p'}=1$.

We recall some standard definitions and facts about measure spaces. A measurable space
is called a \emph{standard Borel space} when
it is endowed with the $\sigma$-algebra of Borel sets with respect to some Polish topology on the space.
If $(Z,\lambda)$ is a measure space for which $L^p(Z,\lambda)$ is separable, then there exists a
complete $\sigma$-finite measure $\mu$ on a standard Borel space $X$ such
that $L^p(Z,\lambda)$ is isometrically isomorphic to $L^p(X,\mu)$.\\
\ \\
\noindent \textbf{Acknowledgements.} The authors would like to thank Chris Phillips and Nico Spronk for helpful
conversations. We also thank Chris Phillips for sharing some of his unpublished work with us, as well as for
a conversation that led to the result in Section~6.

\section{Algebras of $p$-pseudofunctions on groups}

We will work with norm closed subalgebras of algebras of the form $\B(L^p(X,\mu))$ for fixed $p$ in $[1,\I)$,
a class of operator algebras that was formally introduced in \cite{phillips crossed products}. We will recall
in this section the basic definitions and results related to $L^p$-operator algebras, with focus on the
$L^p$-operator algebras associated with locally compact groups; the so-called algebras of $p$-pseudofunctions.

\begin{df}(See Definition 1.1 in \cite{phillips crossed products}).
Let $A$ be a Banach algebra and let $p\in [1,\I)$. We say that $A$ is an \emph{$L^p$-operator algebra} if there
exist a measure space $(X,\mu)$ and an isometric homomorphism $A\to \B(L^p(X,\mu))$.\end{df}

We point out that an $L^2$-operator algebra is an operator algebra in the usual (nonselfadjoint) sense.
\newline

Relevant examples of $L^p$-operator algebras are those associated to a locally compact group, whose
definition we proceed to recall. The algebra $F^p_\lambda(G)$ of $p$-pseudofunctions of a locally
compact group $G$ was introduced by Herz in
\cite{herz:HarmSynth} (we are thankful to Y. Choi and M. Daws for calling our attention to this reference),
where it was denoted by $PF_p(G)$. (This algebra is also called ``reduced group
$L^p$-operator algebra" in \cite{phillips crossed products}.) The full group $L^p$-operator algebra $F^p(G)$
of $G$ was introduced in \cite{phillips crossed products}.

\begin{df} \label{df: LpGpAlgs}
Let $G$ be a locally compact group and let $p\in [1,\I)$. Denote by $\mathrm{Rep}_{\mathcal{L}^p}(G)$ the class
of all contractive representations of $L^1(G)$ on separable $L^p$-spaces. We denote by
$F^p(G)$ the completion of $L^1(G)$ in the norm given by
\[\|f\|_{\mathcal{L}^p}=\sup\left\{\|\pi(f)\|\colon \pi\in \mathrm{Rep}_{\mathcal{L}^p}(G)\right\}\]
for $f\in L^1(G)$. We call $F^p(G)$ the \emph{full group $L^p$-operator algebra} of $G$.

The \emph{algebra of $p$-pseudofunctions} on $G$ (or \emph{reduced group $L^p$-operator algebra} of $G$),
denoted by $F^p_\lambda(G)$, is the completion of $L^1(G)$ in the norm
\[\|f\|_{F^p_\lambda(G)}=\|\lambda_p(f)\|_{\B(L^p(G))}\]
for $f\in L^1(G)$.
\end{df}

Algebras of pseudofunctions on groups were studied by a number of authors since their introduction by Herz.
See, for example, \cite{herz:HarmSynth}, \cite{neufang runde}, \cite{Runde:QSL},
\cite{phillips crossed products}, \cite{GarThi_GpsLp} and \cite{GarThi_functoriality}.

Adopt the notation of Definition \ref{df: LpGpAlgs}.
By universality of $F^p(G)$, there exists a natural contractive homomorphism
$\kappa\colon F^p(G)\to F^p_\lambda(G)$.

The following is part of Theorem~3.7 in \cite{GarThi_GpsLp}. We point out that the same result was
obtained independently by Phillips (\cite{phillips crossed products},\cite{Phi:multiplic}),
using different methods.

\begin{thm}
\label{thm:AmenTFAE}
Let $G$ be a locally compact group, and let $p\in(1,\infty)$.
The following are equivalent:
\begin{enumerate}
\item
The group $G$ is amenable.
\item
The canonical map $\kappa\colon F^p(G)\to F^p_\lambda(G)$ is an isometric isomorphism.
\end{enumerate}
\end{thm}

In view of Theorem \ref{thm:AmenTFAE}, we will not distinguish between $F^p(G)$ and $F^p_\lambda(G)$
whenever $G$ is an amenable group.
\newline

We now turn to algebras of $p$-pseudofunctions of finite cyclic groups.

\begin{eg} \label{eg: FpZn}
Let $n$ in $\N$ and let $p\in [1,\I)$. Consider the group $L^p$-operator algebra $F^p(\Z_n)$ of $\Z_n$,
which is the Banach subalgebra of $\B(\ell^p_n)$ generated by the cyclic shift $s_n$ of order $n$,
\[s_n= \left( \begin{array}{ccccc}
0 &  &  &  & 1 \\
1 & 0 & &  &  \\
& \ddots & \ddots &  &  \\
&  & \ddots & 0 &  \\
&  &  & 1 & 0 \end{array} \right).\]

It is easy to check that $F^p(\Z_n)$ is isomorphic as a complex algebra to $\C^n$,
but the canonical embedding $\C^n\cong F^p(\Z_n)\hookrightarrow M_n$ is not as diagonal matrices.
The norm on $F^p(\Z_n)$ can be computed as follows. Set
$\omega_n=e^{\frac{2\pi i}{n}}$, and set
\[u_n= \frac{1}{\sqrt{n}} \left( \begin{array}{ccccc}
1 & 1 & 1 & \cdots & 1 \\
1 &\omega_n &\omega_n^2 & \cdots &\omega_n^{n-1} \\
1 &\omega_n^2 &\omega_n^4 & \cdots &\omega_n^{2(n-1)} \\
\vdots & \vdots & \vdots & \ddots & \vdots \\
1 &\omega_n^{n-1} &\omega_n^{2(n-1)} & \cdots &\omega_n^{(n-1)^2} \end{array} \right).\]
If $\xi=(\xi_1,\ldots,\xi_n)\in \C^n$, then its norm as an element in $F^p(\Z_n)$ is
\[ \|\xi\|_{F^p(\Z_n)}=\left\| u_n \left( \begin{array}{ccccc}
\xi_1 &  &  &    \\
 & \xi_2 & &   \\
 &  & \ddots &    \\
&  &  &  \xi_n  \end{array} \right)u_n^{-1}\right\|_p.\]

 he matrix $u_n$ is a unitary (in the sense that its conjugate transpose is its inverse), and hence $\|\xi\|_{F^2(\Z_n)}=\|\xi\|_\infty$. On the other hand, if $1\leq p\leq q\leq 2$, or if
$2\leq q\leq p<\I$,
then $\|\cdot\|_{F^q(\Z_n)}\leq \|\cdot\|_{F^{p}(\Z_n)}$, by Corollary~3.20 in \cite{GarThi_GpsLp}.
In particular, the norm $\|\cdot\|_{F^p(\Z_n)}$ always dominates the norm
$\|\cdot\|_\I$.\end{eg}

We recall a number of facts from \cite{GarThi_GpsLp} and \cite{GarThi_functoriality}, specialized
to the case of finite cyclic groups, which will be needed in the following sections.

The following is a consequence of Proposition~2.3 in \cite{GarThi_functoriality}.

\begin{prop} \label{prop:inclusion subgp}
Let $n$ in $\N$ and let $p$ in $[1,\I)$. If $k$ divides $n$, then the canonical
inclusion $\Z_k\to \Z_n$ induces an isometric, unital homomorphism
\[F^p(\Z_k)\to F^p(\Z_n).\]\end{prop}

Given $n\in\N$, we denote by $\tau_n\colon\C^n\to\C^n$ the cyclic shift, which is given
by
$$\tau_n(x_0,x_1,\ldots,x_{n-2},x_{n-1})=(x_{n-1},x_0,\ldots,x_{n-3},x_{n-2})$$
for all $(x_0,\ldots,x_{n-1})$ in $\C^n$.

The following is Proposition~3.2 in \cite{GarThi_functoriality}.

\begin{prop} \label{prop: shift invariance norm on FpZn}
Let $n$ in $\N$ and let $p$ in $[1,\I)$. Then $\tau_n\colon F^p(\Z_n)\to F^p(\Z_n)$ is an isometric isomorphism.
\end{prop}

Let $k$ and $n$ be positive integers. For each $r$ in $\{0,\ldots,k-1\}$, we define a restriction map
$$\rho^{(nk\to n)}_r\colon\C^{nk}\to\C^n$$
by sending a $nk$-tuple $\beta$ to the $n$-tuple
\[
\rho^{(nk\to n)}_r(\beta)_q=\beta_{qk+r},\quad q=0,\ldots,n-1.
\]
The following lemma asserts that $\rho^{(nk\to n)}_r$ is contractive when regarded as a map $F^p(\Z_{nk})\to F^p(\Z_n)$.

\begin{lemma}\label{lma: restriction Z_nk}
Let $k,n\in\N$, and let $p\in[1,\infty)$.
For each $r\in\{0,\ldots,k-1\}$, the restriction map $\rho^{(nk\to n)}_r$ is a contractive, unital homomorphism $F^p(\Z_{nk})\to F^p(\Z_n)$.
\end{lemma}
\begin{proof}
Let $\tau\colon F^p(\Z_{nk})\to F^p(\Z_{nk})$ be the cyclic shift.
Then $\tau$ is an isometric isomorphism by Proposition \ref{prop: shift invariance norm on FpZn}.
Note that $\rho^{(nk\to n)}_r=\rho^{(nk\to n)}_0\circ \tau^r$.
Thus, it is enough to show that $\rho^{(nk\to n)}_0$ is a contractive, unital homomorphism.
This follows immediately from Proposition~2.4 in \cite{GarThi_functoriality},
so the proof is finished.
\end{proof}

It is a well-known fact that for $p\in [1,\I)\setminus\{2\}$, the only $n$ by $n$ matrices
that are isometries when regarded as linear maps $\ell^p_n\to \ell^p_n$, are precisely
the complex permutation matrices. These are the matrices all of whose entries are either
zero or a complex number of modulus one, that have exactly one non-zero entry on each column
and each row.

Using the above mentioned fact, the proof of the following proposition is straightforward,
using the description of the norm on $F^p(\Z_n)$ given in Example \ref{eg: FpZn}.

\begin{prop}\label{prop:IsomFpZn}
Let $n\in\N$, let $p\in [1,\I)\setminus\{2\}$. Set $\omega_n=e^{\frac{2\pi i}{n}}$.
If $x\in F^p(\Z_n)$ is invertible and satisfies $\|x\|=\|x^{-1}\|=1$, then
there exist $\zeta\in S^1$ and $k\in \{0,\ldots,n-1\}$ such that
\[x=\zeta\cdot \left(1,\omega_n^k,\ldots,\omega_n^{(n-1)k}\right).\]
The converse also holds.\end{prop}

Proposition \ref{prop:IsomFpZn} admits a significant generalization, as follows.
If $G$ is a locally compact group, $p\in [1,\I)\setminus\{2\}$, and $x$ is an invertible
isometry in the multiplier algebra of $F^p_\lambda(G)$, then there exist $\zeta\in S^1$
and $g\in G$ such that $x=\zeta \delta_{g}$, where $\delta_g$ denotes the point mass
measure concentrated on $\{g\}$. This result will be proved in \cite{GarThi_IsomConvAlgs}, and will be
used tl\o prove that for two locally compact groups $G_1$ and $G_2$, and for any
$p_1,p_2 \in [1,\I)\setminus\{2\}$, then $F^{p_1}_\lambda(G_1)$ is isometrically isomorphic
to $F^{p_2}_\lambda(G_2)$ if and only if $G_1$ is isomorphic to $G_2$, and either
$p_1=p_2$ or $\frac{1}{p_1}+\frac{1}{p_2}=1$.
\newline

We close this section with an easy fact that will be crucial in our proof of
Theorem \ref{thm: SpConfMapsAlgs}.

\begin{prop}\label{prop:RegExp}
Let $n,d\in\N$ with $d|n$ and $d<n$, and let $p\in [1,\I)\setminus\{2\}$. There exists $\alpha\in F^p(\Z_{n})$
such that
\[\|\alpha\|_{F^p(\Z_{n})}> \sup_{b=0,\ldots,\frac{n}{d}-1} \left\|\rho^{(n\to d)}_b(\alpha)\right\|_{F^p(\Z_d)}.\]
\end{prop}
\begin{proof} Let $\omega=e^{\frac{2\pi i}{n}}$ and set
\[\beta=\left(1,\ldots,1,\omega^d,\ldots,\omega^d,\ldots,\omega^{nd-d},\ldots,\omega^{nd-d}\right),\]
regarded as an element in $F^p(\Z_n)$. (There are $\frac{n}{d}$ repetitions of each power of $\omega$.)
Then $\beta$ is invertible (the inverse being its coordinate-wise complex conjugate).

We claim that
\[\sup_{b=0,\ldots,\frac{n}{d}-1} \left\|\rho^{(n\to d)}_b(\beta)\right\|_{F^p(\Z_d)}=
\sup_{b=0,\ldots,\frac{n}{d}-1} \left\|\rho^{(n\to d)}_b(\beta^{-1})\right\|_{F^p(\Z_d)}=1.\]

First, note that $\rho^{(n\to d)}_b(\beta)=\rho^{(n\to d)}_a(\beta)$ and
$\rho^{(n\to d)}_b(\beta^{-1})=\rho^{(n\to d)}_a(\beta^{-1})$ for all $a,b=0,\ldots,\frac{n}{d}-1$.
Since $\rho^{(n\to d)}_0(\beta)=\left(1,\omega^d,\ldots,\omega^{nd-d}\right)$ is the canonical invertible isometry
generating $F^p(\Z_d)$, and $\rho^{(n\to d)}_0(\beta^{-1})$ is its inverse,  we conclude that
\[\left\|\rho^{(n\to d)}_0(\beta)\right\|_{F^p(\Z_d)}=\left\|\rho^{(n\to d)}_0(\beta)\right\|_{F^p(\Z_d)}=1,\]
and the claim follows.

We claim that either $\|\beta\|_{F^p(\Z_{n})}>1$ or $\|\beta^{-1}\|_{F^p(\Z_{n})}>1$.

Based on the description of the invertible isometries of $F^p(\Z_n)$ given in Proposition \ref{prop:IsomFpZn},
it is clear that $\beta$ is not an invertible isometry, so not both $\beta$ and $\beta^{-1}$ have norm
one. Since $\|\cdot\|_{F^p(\Z_n)}\geq \|\cdot\|_\I$ (see Example \ref{eg: FpZn}) and $\|\beta\|_\I=\|\beta^{-1}\|_\I=1$, the claim follows.

The result now follows by setting $\alpha$ equal to either $\beta$ or $\beta^{-1}$, as appropriate.
\end{proof}

\section{Spectral configurations}

In this section, we study a particular class of commutative Banach algebras. They are naturally associated to
certain sequences of subsets of $S^1$ that we call \emph{spectral configurations}; see Definition \ref{df: spectral conf}.
We show that all such Banach algebras are generated by an invertible isometry of an $L^p$-space together with its inverse.
(In fact, the invertible isometry can be chosen to act on $\ell^p$.)
We also show that there is a strong dichotomy with respect to the
isomorphism type of these algebras: they are isomorphic to either $F^p(\Z)$, or to the space of all continuous
functions on its maximal ideal space; see Theorem \ref{thm: Fpsigma}.
In the last part of the section, we study when two spectral configurations give rise to isometrically isomorphic Banach algebras.

We mention here that one of the main results in Section 5, Theorem \ref{thm: description}, states that the Banach algebra generated
by an invertible isometry of an $L^p$-space together with its inverse, is isometrically isomorphic to the $L^p$-operator algebra
associated to a spectral configuration which is naturally associated to the isometry.
\\

We begin by defining a family of norms on algebras of the form $C(\sigma)$, where $\sigma$ is a certain closed subset of $S^1$.
We will later see that these norms are exactly those that arise from spectral configurations consisting of exactly one nonempty set.
Recall that if $n$ is a positive integer, we denote $\omega_n=e^{\frac{2\pi i}{n}}\in S^1$.

\begin{df} \label{df: norm sigma n}
Let $p\in [1,\I)$ and let $n$ in $\N$. Let $\sigma$ be a nonempty closed subset of
$S^1$ which is invariant under rotation by $\omega_n$. For $f$ in $C(\sigma)$, we define
\[\|f\|_{\sigma,n,p}=\sup_{t\in \sigma}\left\|\left(f(t),f(\omega_n t),\ldots,f(\omega_n^{n-1}t)\right)\right\|_{F^p(\Z_n)}.\]
\end{df}

We will see in Proposition \ref{prop: norm sigma n} that, as its notation suggests, the function $\|\cdot\|_{\sigma,n,p}$ is
indeed a norm on $C(\sigma)$.\\
\indent When $p=2$, we have $\|\cdot\|_{F^p(\Z_n)}= \|\cdot\|_\infty$ and hence the norm $\|\cdot\|_{\sigma,n,2}$ is the
supremum norm $\|\cdot\|_\infty$ for all $n$ in $\N$ and every closed subset $\sigma\subseteq S^1$. On the other hand,
if $n=1$ then $F^p(\Z_1)\cong \C$ with the usual norm, so $\|\cdot\|_{\sigma,1,p}$ is the supremum norm for every $p$ in
$[1,\I)$ and every closed subset $\sigma\subseteq S^1$.
However, the algebra $(C(\sigma),\|\cdot\|_{\sigma,n,p})$ is never isometrically isomorphic to $(C(\sigma),\|\cdot\|_\I)$
when $p\neq 2$ and $n>1$; see part (5) of Theorem \ref{thm: Fpsigma}.

\begin{prop}\label{prop: norm sigma n}
Let $p\in [1,\I)$, let $n$ in $\N$, and let $\sigma$ be a nonempty closed subset of $S^1$ which is invariant under
rotation by $\omega_n$.
\be
\item
The function $\|\cdot\|_{\sigma,n,p}$ is a norm on $C(\sigma)$.
\item
The norm $\|\cdot\|_{\sigma,n,p}$ is equivalent to $\|\cdot\|_\infty$.
\item
The Banach algebra $\left(C(\sigma),\|\cdot\|_{\sigma,n,p}\right)$ is isometrically representable on $\ell^p$, and
hence it is an $L^p$-operator algebra.
\ee
\end{prop}
\begin{proof}
(1).
This follows immediately from the fact that $\|\cdot\|_{F^p(\Z_n)}$ is a norm.

(2).
Since $\|\cdot\|_\infty\leq \|\cdot\|_{F^p(\Z_n)}$, one has $\|\cdot\|_\infty\leq \|\cdot\|_{\sigma,n,p}$.
On the other hand, since $F^p(\Z_n)$ is finite dimensional, there exists a (finite) constant $C=C(n,p)$ such that $\|\cdot\|_{F^p(\Z_n)}\leq
C \|\cdot\|_\infty$.
Thus, for $f\in C(\sigma)$, we have
\begin{align*}
\|f\|_{\sigma,n,p} & =\sup_{t\in \sigma}\left\|\left(f(t),f(\omega_n t),\ldots,f(\omega_n^{n-1}t)\right)\right\|_{F^p(\Z_n)}\\
 &\leq C \sup_{t\in \sigma}\left\|\left(f(t),f(\omega_n t),\ldots,f(\omega_n^{n-1}t)\right)\right\|_{\infty}\\
&= C\|f\|_\infty.
\end{align*}
We conclude that $\|\cdot\|_\infty\leq \|\cdot\|_{\sigma,n,p}\leq C \|\cdot\|_\infty$, as desired.

(3).
Denote by $u_n\in M_n$ the matrix displayed in Example \ref{eg: FpZn}.
Let $(t_k)_{k\in\N}$ be a dense sequence in $\sigma$, and consider the linear map
\[
\rho\colon C(\sigma)\to \B\left(\bigoplus_{k\in \N}\ell^p_{n}\right)
\]
given by
\[
\rho(f)=\bigoplus_{k\in\N} u_n\left( \begin{array}{ccccc}
f(t_k) &  &  &    \\
 & f(\omega_n t_k) & &   \\
 &  & \ddots &    \\
&  &  &  f(\omega_n^{n-1} t_k)  \end{array} \right)u_n^{-1}
\]
for $f\in C(\sigma)$.
It is easy to verify that $\rho$ is a homomorphism.
Moreover, given $f$ in $C(\sigma)$, we use the description of the norm
$\|\cdot\|_{F^p(\Z_n)}$ from Example \ref{eg: FpZn} at the second step, and continuity of $f$ at the last step, to get
\begin{align*}
\|\rho(f)\|&= \sup_{k\in\N} \left\| u_n\mbox{diag}(f(t_k), f(\omega_n t_k),\ldots,f(\omega_n^{n-1} t_k))u_n^{-1}\right\|\\
&= \sup_{k\in\N}\left\|\left(f(t_k),f(\omega_n t_k),\ldots,f(\omega_n^{n-1}t)\right)\right\|_{F^p(\Z_n)}\\
&=\|f\|_{\sigma,n,p}.
\end{align*}
We conclude that $\rho$ is isometric, as desired.
\end{proof}

Now that we have analyzed the basic example of a spectral configuration, we proceed to study the general case.

\begin{df} \label{df: spectral conf}
A \emph{spectral configuration} is a sequence $\sigma=(\sigma_n)_{n\in\NI}$ of closed subsets of $S^1$ such that
\be\item For every $n\in\N$, the set $\sigma_n$ is
invariant under rotation by $\omega_n$;
\item The set $\sigma_\I$ is either empty or all of $S^1$; and
\item We have $\sigma_n\neq \emptyset$ for at least one $n$ in $\NI$.\ee

The \emph{order} of the spectral configuration $\sigma$ is defined as
\[\mbox{ord}(\sigma)=\sup\ \{n\in\NI\colon \sigma_n\neq\emptyset\}.\]\end{df}

Note that a spectral configuration may have infinite order and yet consist of only finitely many nonempty sets.

We adopt the convention that for $p\in [1,\I)$, the function $\|\cdot\|_{\sigma_\I,\I,p}$ is the zero function if
$\sigma_\I=\emptyset$, and the norm of $F^p(\Z)$ otherwise.

\begin{df}\label{df: Fpsigma}
Let $\sigma=(\sigma_n)_{n\in\NI}$ be a spectral configuration and let $p\in [1,\I)$.
Set
\[
\overline{\sigma}=\overline{\bigcup_{n\in\NI}\sigma_n}\subseteq S^1,
\]
and for $f$ in $C(\overline{\sigma})$, define
\[
\|f\|_{\sigma,p}= \sup_{n\in\NI} \|f|_{\sigma_n}\|_{\sigma_n,n,p}.
\]
The normed algebra associated to $\sigma$ and $p$ is
\[
F^p(\sigma)=\left\{f\in C(\overline{\sigma})\colon \|f\|_{\sigma,p}<\I\right\},
\]
endowed with the norm $\|\cdot\|_{\sigma,p}$.
\end{df}

Since the H\"older exponent $p\in [1,\I)$ will be clear from the context (in particular, it is included in the notation for
$F^p(\sigma)$), we will most of the times from it from the notation for the norm $\|\cdot\|_{\sigma,p}$,
and write $\|\cdot\|_\sigma$ instead, except when confusion
is likely to arise.

\begin{thm} \label{thm: Fpsigma}
Let $\sigma=(\sigma_n)_{n\in\NI}$ be a spectral configuration and let $p\in [1,\I)$.
\be
\item
The Banach algebra $F^p(\sigma)$ is an $L^p$-operator algebra that can be represented on $\ell^p$.
\item
The Banach algebra $F^p(\sigma)$ is generated by an invertible isometry together with its inverse.
\item
If $\mbox{ord}(\sigma)=\I$, then there is a canonical isometric isomorphism
$$F^p(\sigma)\cong F^p(\Z).$$
In particular, if $p\neq 2$, then not every continuous function on $\overline{\sigma}$ has finite $\|\cdot\|_{\sigma,p}$-norm.
\item
If $\mbox{ord}(\sigma)=N<\I$, then $\|\cdot\|_\sigma = \max\limits_{n=1,\ldots,N}\|\cdot\|_{\sigma_n,n}$.
Moreover, the identity map on $C(\overline{\sigma})$ is a canonical Banach algebra isomorphism
\[
(F^p(\sigma),\|\cdot\|_\sigma)\cong \left(C(\overline{\sigma}),\|\cdot\|_\infty\right),
\]
and thus every continuous function on $\overline{\sigma}$ has finite $\|\cdot\|_\sigma$-norm.
\item For $p\in [1,\I)\setminus\{2\}$, there is a canonical isometric isomorphism
\[
(F^p(\sigma),\|\cdot\|_\sigma)\cong \left(C(\overline{\sigma}),\|\cdot\|_\infty\right),
\]
if and only if $\mbox{ord}(\sigma)=1$. When $p=2$, such an isometric isomorphism always exists.
\ee
\end{thm}
\begin{proof}
(1).
For every $n$ in $\N$, use part (3) in Proposition
\ref{prop: norm sigma n} to find an isometric representation $\rho_n\colon C(\sigma_n)\to \B(\ell^p),$ and let $\rho_\I\colon
F^p(\Z)\to \B(\ell^p)$ be the canonical isometric representation. Define
$$\rho\colon F^p(\sigma)\to \B\left(\bigoplus_{n\in\NI}\ell^p\right) \ \ \mbox{ by } \ \
\rho(f)=\bigoplus_{n\in\NI} \rho_n(f|_{\sigma_n})$$
for all $f$ in $F^p(\sigma)$. It is immediate to check that $\rho$ is isometric. Since $\bigoplus\limits_{n\in\NI}\ell^p$ is
isometrically isomorphic to $\ell^p$, the result follows.

(2).
Let $\iota \colon \overline{\sigma}\to \C$ be the inclusion map, and let $\iota^{-1}\colon \overline{\sigma}\to\C$ be
its (pointwise) inverse.
It is clear that $\|\iota\|_\sigma=\|\iota^{-1}\|_\sigma=1$.
Thus, $\iota$ and $\iota^{-1}$ are invertible isometries, and they clearly generate $F^p(\sigma)$.

(3).
It is enough to show that for every $f$ in $\C[\Z]$, one has
$$\|f(\iota)\|_{F^p(\sigma)}=\|\lambda_p(f)\|_{\B(\ell^p)}.$$
Since $\iota$ is an invertible isometry, the universal property of $F^p(\Z)$ implies that $\|\lambda_p(f)\|\geq \|f(\iota)\|$.

If $\sigma_\I\neq\emptyset$, then
\[\|f(\iota)\|_\sigma\geq \|f(\iota)\|_{\sigma_\I,\I}=\|f\|_{F^p(\Z)},\]
and the result follows. We may therefore assume that $\sigma_\I=\emptyset$.\\
\indent In order to show the opposite inequality, let $\ep>0$ and choose an element $\xi=(\xi_k)_{k\in\Z}$ in $\ell^p$ of finite
support with $\|\xi\|_p^p=1$ and such that
$$\|\lambda_p(f)\xi\|_p>\|\lambda_p(f)\|-\ep.$$
Choose $K$ in $\N$ such that $\xi_k=0$ whenever $|k|>K$. Find a positive integer $M$ in $\N$ and complex coefficients $a_m$ with
$-M\leq m\leq M$ such that $f(x,x^{-1})=\sum\limits_{m=-M}^Ma_mx^m$. Since $\mbox{ord}(\sigma)=\I$, there exists $n>2(K+M)$ such that
$\sigma_n\neq\emptyset$. Fix $t$ in $\sigma_n$ and define a representation $\rho\colon F^p(\sigma) \to \B(\ell^p_n)$ by
\[\rho(h)=\left( \begin{array}{ccccc}
0 &  &  &  &  h(\omega_n^{n-1}t) \\
h(t) & 0 & &  &  \\
& \ddots & \ddots &  &  \\
&  & \ddots & 0 &  \\
&  &  & h(\omega_n^{n-2}t) & 0 \end{array} \right)\]
for all $h$ in $F^p(\sigma)$. It is clear that
$$\|\rho(h)\|=\|(h(t),h(\omega_nt),\ldots,h(\omega_n^{n-1}t))\|_{F^p(\Z_n)}\leq \|h|_{\sigma_n}\|_{\sigma_n,n}\leq \|h\|_\sigma,$$
so $\rho$ is contractive. Denote by $e_0$ the basis vector $(1,0,\ldots,0)\in \ell^p_n$. Set $v=\rho(\iota)$, and note that
$v^k(e_0)=e_{n-k}\omega_n^{n-k}t$ for all $k=0,\ldots,n-1$, where indices are taken modulo $n$.\\
\indent Let $\eta\in \ell^p_n$ be given by $\eta=\sum\limits_{k=-K}^K\xi_ke_{n-k}\omega_n^{n-k}t$. Then $\|\eta\|_p=\|\xi\|_p=1$.
Moreover,
\begin{align*}\rho(f(\iota))\eta &=\sum_{m=-M}^M a_mv^m\left(\sum_{k=-K}^K\xi_ke_{n-k}\omega_n^{n-k}\right)\\
 &=\sum_{m=-M}^M\sum_{k=-K}^Ka_m\xi_k e_{n-m-k}\omega_n^{n-m-k}t\\
&=\sum_{j=-M-K}^{N+K}(\lambda_p(f)\xi)_jv^j(e_0)
\end{align*}

The elements $e_{n-m-k}$ for $-M\leq m\leq M$ and $-K\leq k\leq K$ are pairwise distinct, by the choice of $n$. We
use this at the first step to get
\begin{align*} \|\rho(f(\iota))\eta\|^p_p &= \sum_{j=-N-K}^{N+K}\left|\left[\lambda_p(f)\xi\right]_j\right|^p \left\|v^{j}(e_0)\right\|
=\left\|\lambda_p(f)\xi\right\|^p_p.
\end{align*}
We deduce that
\begin{align*} \|f(\iota)\|\geq \|\rho(f(\iota))\|\geq \|\rho(f(\iota))\eta\|_p=\left\|\lambda_p(f)\xi\right\|_p>\|\lambda_p(f)\|-\ep.\end{align*}
Since $\ep>0$ is arbitrary, we conclude that $\|f(\iota)\|_{F^p(\sigma)}\geq \|\lambda_p(f)\|_{F^p(\Z)}$, as desired.

(4).
Let $N=\mbox{ord}(\sigma)$. It is immediate from the definition that $\|\cdot\|_\sigma =
\max\limits_{n=1,\ldots,N}\|\cdot\|_{\sigma_n,n}$.
Moreover, for each $n=1,\ldots,N$, use part (2) in Proposition \ref{prop: norm sigma n} to find a constant $C(n)>0$ satisfying
$$\|\cdot\|_\I\leq \|\cdot\|_{\sigma_n,n}\leq C(n)\|\cdot\|_\I.$$
Set $C=\max\{C(1),\ldots,C(N)\}$. Then $\|\cdot\|_\I\leq \|\cdot\|_{\sigma}\leq C\|\cdot\|_\I,$ and thus
$\|\cdot\|_\sigma$ is equivalent to $\|\cdot\|_\I$, as desired.

(5).
It is clear from the comments after Definition \ref{df: norm sigma n} that $\|\cdot\|_\sigma=\|\cdot\|_\I$ if either
$p=2$ or $\mbox{ord}(\sigma)=1$. Conversely, suppose that $p\neq 2$ and that $\mbox{ord}(\sigma)>1$.
Choose $n$ in $\NI$ with $n>1$ such that $\sigma_n\neq\emptyset$.
If $n=\I$, then $F^p(\sigma)$ is isometrically isomorphic to $F^p(\Z)$ by part (3) of this theorem.
The result in this case follows from part (2) of Corollary~3.20 in \cite{GarThi_GpsLp} for $G=\Z$ and $p'=2$.
We may therefore assume that $n<\I$.

Let $t$ in $\sigma_n$ and choose a continuous function
$f$ on $\overline{\sigma}$ with $\|f\|_\I=1$ such that
\[\left\|\left(f(t),f(\omega_nt),\ldots,f(\omega_n^{n-1})\right)\right\|_{F^p(\Z_n)}>1.\]
(See, for example, the proof of Proposition \ref{prop:RegExp}.)
We get
\begin{align*}
\|f\|_\sigma &\geq \|f|_{\sigma_n}\|_{\sigma_n,n}\\
 &\geq \|(f(t),f(\omega_nt),\ldots,f(\omega_n^{n-1}t))\|_{F^p(\Z_n)}\\
&>1=\|f\|_\I.
\end{align*}
It follows that $\|\cdot\|_\sigma$ is not the supremum norm, and the claim is proved.
\end{proof}

We now turn to the question of when two spectral configurations determine the same $L^p$-operator algebra.\\
\indent It turns out that $F^p(\sigma)$ does not in general determine $\sigma$, and thus there are several spectral configurations
whose associated $L^p$-operator algebras are pairwise isometrically isomorphic. For example, let $\sigma$ be the spectral
configuration given by $\sigma_n=S^1$ for all $n$ in $\NI$, and let $\tau$ be given by $\tau_{n}=\Z_{n}$ for $n$ in $\N$ and
$\sigma_\I=\emptyset$. Then $F^p(\sigma)\cong F^p(\tau)\cong F^p(\Z)$ by part (3) of Theorem \ref{thm: Fpsigma}.\\
\indent However, as we will see later, given $F^p(\sigma)$ one can recover what we call the \emph{saturation} of $\sigma$.
See Corollary \ref{cor: isom classif of Fpsigma}.

\begin{df}
\label{df:saturated}
Let $\sigma$ be a spectral configuration.
We define what it means for $\sigma$ to be saturated in two cases, depending on whether its order is infinite or not.
\begin{itemize}
\item
If $\mbox{ord}(\sigma)=\I$, we say that $\sigma$ is \emph{saturated} if $\sigma_n=S^1$ for all $n$ in $\NI$.
\item
If $\mbox{ord}(\sigma)<\I$, we say that $\sigma$ is \emph{saturated} if $\sigma_m\subseteq \sigma_n$ whenever $n$ divides $m$.
\end{itemize}
\end{df}

\begin{rems}
Let $\sigma$ be a saturated spectral configuration.

(1)
The set $\sigma_\infty$ is nonempty (in which case it equals $S^1$) if and only if $\mbox{ord}(\sigma)=\I$. Since
the order of $\sigma$ is determined by the spectral sets $\sigma_n$ for $n$ finite, we conclude that the
spectral set $\sigma_\I$ is redundant.

(2)
For every $n$ in $\overline{\N}$, we have $\sigma_n\subseteq\sigma_1$: for $n<\I$, this is true since 1 divides $n$, and for
$n=\I$ it is true by definition. In particular, since $\sigma_1$ is
closed, we must have $\sigma_1=\overline{\sigma}$.
\end{rems}

Denote by $\Sigma$ the family of all saturated spectral configurations.
We define a partial order on $\Sigma$ by setting $\sigma\leq\tau$ if $\sigma_n\subseteq\tau_n$
for every $n$ in $\NI$.

\begin{lemma}
The partial order defined above turns $\Sigma$ into a complete lattice.
Moreover, $\Sigma$ has a unique maximal element, and its minimal elements are in bijection with $S^1$.
\end{lemma}
\begin{proof}
Let $\Omega$ be a nonempty subset of $\Sigma$.
For each $n\in\overline{\N}$, set
\[
(\sup\Omega)_n=\overline{\bigcup_{\sigma\in\Omega}\sigma_n} \ \ \mbox{ and } \ \ (\inf\Omega)_n=\bigcap_{\sigma\in\Omega}\sigma_n.
\]
It is readily verified that this defines elements $\sup\Omega$ and $\inf\Omega$ of $\Sigma$ that are the supremum and infimum of $\Omega$, respectively.

The unique maximal element $\sigma^\I$ is given by $\sigma^\I_n=S^1$ for all $n\in\NI$.
For each element $\zt$ in $S^1$, there is a minimal configuration $\sigma(\zt)$ given by $\sigma(\zt)_1=\{\zt\}$ and $\sigma_n(\zt)=\emptyset$ for $n\geq 2$.
Conversely, if $\sigma$ is a saturated configuration, then $\sigma_1$ is not empty, so we may choose $\zt\in \sigma_1$.
Then $\sigma(\zt)$ is a minimal configuration and $\sigma(\zt)\leq \sigma$.
\end{proof}

\begin{df}
Let $\sigma$ be a spectral configuration.
The \emph{saturation} of $\sigma$, denoted by $\widetilde{\sigma}$, is the minimum of all saturated spectral configurations that contain $\sigma$.
\end{df}

Note that any spectral configuration (saturated or not) is contained in the maximal configuration $\sigma^\I$. It follows
that the saturation of a spectral configuration is well defined.

The proof of the following lemma is a routine exercise and is left to the reader.

\begin{lemma}
Let $\sigma$ be a spectral configuration and let $\widetilde\sigma$ be its saturation.
Then:
\begin{enumerate}
\item
We have $\overline{\sigma} =\overline{\widetilde{\sigma}} =\widetilde{\sigma}_1$.
\item
We have $\mbox{ord}(\sigma)=\mbox{ord}(\widetilde{\sigma})$.
\item
If $\mbox{ord}(\sigma)<\I$, then $\widetilde{\sigma}_n =\bigcup\limits_{k=1}^\infty\sigma_{kn}$ for every $n$ in $\N$.
\end{enumerate}
\end{lemma}

If $\sigma$ is a spectral configuration, we denote by $\iota_\sigma\in F^p(\sigma)\subseteq C(\overline{\sigma})$ the canonical
inclusion of $\overline{\sigma}$ into $\C$.

\begin{prop}
Let $\sigma$ be a spectral configuration and let $\widetilde\sigma$ be its saturation.
Then there is a canonical isometric isomorphism
\[
F^p(\sigma)\cong F^p(\widetilde\sigma),
\]
which sends $\iota_\sigma$ to $\iota_{\widetilde\sigma}$.
\end{prop}
\begin{proof}
If $\sigma$ has infinite order, then the result follows from part (3) of Theorem \ref{thm: Fpsigma}.

Assume now that $\sigma$ has finite order.
Then the underlying complex algebra of both $F^p(\sigma)$ and $F^p(\widetilde{\sigma})$ is $C(\overline{\sigma})$.
We need to show that the norms $\|\cdot\|_\sigma$ and $\|\cdot\|_{\widetilde{\sigma}}$ coincide.
Let $f\in C(\overline{\sigma})$.
We claim that $\|f\|_\sigma=\|f\|_{\widetilde{\sigma}}$.
Since $\sigma_n\subseteq\widetilde{\sigma}_n$, it is immediate that $\|f\|_\sigma\leq \|f\|_{\widetilde{\sigma}}$.

For the reverse inequality, let $n\in\N$.
By Lemma \ref{lma: restriction Z_nk}, for every $k\in\N$ the restriction map $\rho^{(nk\to k)}_0\colon F^p(\Z_{nk})\to F^p(\Z_{n})$ is contractive.
Using this at the fourth step, we obtain
\begin{align*}
\|f\|_{\widetilde\sigma_n,n}
&= \sup_{t\in\widetilde\sigma_n} \|(f(t),f(\omega_nt),\ldots,f(\omega_n^{n-1}t))\|_{F^p(\Z_n)} \\
&= \sup_{k\in\N} \sup_{t\in\sigma_{nk}}
\|(f(t),f(\omega_nt),\ldots,f(\omega_n^{n-1}t))\|_{F^p(\Z_n)} \\
&= \sup_{k\in\N} \sup_{t\in\sigma_{nk}}
\left\|\rho^{(nk\to k)}_0 (f(t),f(\omega_{nk}t),\ldots,f(\omega_{nk}^{nk-1}t)) \right\|_{F^p(\Z_{nk})} \\
&\leq \sup_{k\in\N} \sup_{t\in\sigma_{nk}}
\|(f(t),f(\omega_{nk}t),\ldots,f(\omega_{nk}^{nk-1}t))\|_{F^p(\Z_{nk})} \\
&\leq \sup_{m\in\N} \|f\|_{\sigma_m,m} = \|f\|_\sigma.
\end{align*}
It follows that $\|f\|_{\widetilde\sigma} =\sup\limits_{n\in\N}\|f\|_{\widetilde\sigma_n,n} \leq\|f\|_\sigma$, as desired.
\end{proof}


\begin{thm}\label{thm: SpConfMapsAlgs}
Let $\sigma$ and $\tau$ be two saturated spectral configurations and let $p\in [1,\I)\setminus\{2\}$.
Then the following conditions are equivalent:
\begin{enumerate}
\item
We have $\tau\leq \sigma$, that is, $\tau_n\subseteq \sigma_n$ for every $n$ in $\NI$.
\item
We have $\tau_1\subseteq \sigma_1$, and
$$\|g_{|\tau_1}\|_\tau\leq \|g\|_\sigma$$
for every $g\in C(\overline\sigma)$.
\item There is a contractive, unital homomorphism
$$\varphi\colon F^p(\sigma)\to F^p(\tau)$$
such that $\varphi(\iota_\sigma)=\iota_\tau$.
\end{enumerate}
\end{thm}
\begin{proof}
The implications `(1) $\Rightarrow$ (2) $\Rightarrow$ (3)' are clear.
For the implication `(3) $\Rightarrow$ (2)', notice that the sets $\sigma_1$ and $\tau_1$ can be canonically identified with
the character spaces of $F^p(\sigma)$ and $F^p(\tau)$, respectively, so that $\sigma_1\supseteq\tau_1$.
Then (2) follows immediately from (3).

Let us show `(2) $\Rightarrow$ (1)'.
For $k\in\N$, define a function $\mu^{(k)}\colon\R\to[0,1]$ as
\[
\mu^{(k)}(x) =\max(0,1-2kx).
\]
These are bump-functions around $0$, with support $\left[-\frac{1}{2k},\frac{1}{2k}\right]$.

Given $n\in\N$, given $\alpha=(\alpha_0,\ldots,\alpha_{n-1})\in\C^n$, and given $t\in S^1$, let us define continuous functions $f_{\alpha,t}^{(k)}\colon S^1\to\C$ for $k\in\N$ as follows:
\[
f_{\alpha,t}^{(k)}(x)
=\sum_{l=0}^{n-1} \alpha_l\cdot\mu^{(kn)}(\dist(x,\omega_n^lt)).
\]
This is a function with $n$ bumps around the points $t,\omega_nt,\ldots,\omega_n^{n-1}t$ taking the values $f_{\alpha,t}^{(k)}(\omega_n^lt)=\alpha_l$ for $l=0,1,\ldots,n-1$.

Let $s\in\sigma_m$ and fix $k\in\N$ such that $\frac{1}{km}<\frac{1}{n}$.
We compute $\left\|f_{\alpha,t}^{(k)}\right\|_{m,s}$.
The support of $f_{\alpha,t}^{(k)}$ is the $\frac{1}{2kn}$-neighborhood of $\{t,t\omega_n,\ldots,t\omega_n^{n-1}\}$.
Assume there are $a\in\{0,\ldots,m-1\}$ and $b\in\{0,\ldots,n-1\}$ such that $\omega_m^as$ belongs to the $\frac{1}{2kn}$-neighborhood of $\omega_n^bt$.

Let $\alpha'$ be the $n$-tuple obtained from $\alpha$ by cyclically rotating by $-b$, that is \[\alpha'=(\alpha_b,\alpha_{b+1},\ldots,\alpha_{n-1},\alpha_0,\ldots,\alpha_{b-1}).\]

Set $t'=\omega_n^bt$.
Then $f_{\alpha,t}^{(k)}=f_{\alpha',t'}^{(k)}$.
Set $s'=\omega_m^as$.
Since the norm on $F^p(\Z_m)$ is rotation-invariant, we have $\|\cdot\|_{m,s}=\|\cdot\|_{m,s'}$.
Thus
\[
\left\|f_{\alpha,t}^{(k)}\right\|_{m,s}
= \left\|f_{\alpha',t'}^{(k)}\right\|_{m,s'}.
\]

We have reduced to the case that
$s'$ is in the $\frac{1}{2kn}$-neighborhood of $t'$.
Let $d=\gcd(n,m)$.
Then $\omega_d^ls'$ is in the $\frac{1}{2kn}$-neighborhood of $\omega_d^lt'$ for each $l=0,\ldots,d-1$.
Since $\frac{1}{kn}<\frac{1}{m}$, the value $f_{\alpha',t'}^{(k)}(\omega_m^rs')$ is zero unless $r$ is a multiple of $\frac{m}{d}$.
Let $\delta=\dist(s',t')$.
Let $r=i\frac{m}{d}+j$ for $i\in\{0,\ldots,d-1\}$ and $j\in\{0,\ldots,\frac{m}{d}-1\}$.
Then
\[
f_{\alpha',t'}^{(k)}(\omega_m^rs')
=\begin{cases}
0, &\text{if } j\neq 0 \\
\mu^{(kn)}(\delta)\alpha'_{i\frac{n}{d}}, &\text{if } j=0 \\
\end{cases}.
\]

We define an inclusion map $\iota^{(d\to m)}\colon\C^d\to\C^m$ as follows.
The tuple $\beta\in\C^d$ is sent to the tuple $\iota^{(d\to m)}(\beta)$ which for $r=i\frac{m}{d}+j$ with $i\in\{0,\ldots,d-1\}$ and $j\in\{0,\ldots,\frac{m}{d}-1\}$ is given by
\[
\iota^{(d\to m)}(\beta)_{r}
=\begin{cases}
0, &\text{if } j\neq 0 \\
\beta_{i}, &\text{if } j=0 \\
\end{cases}.
\]
As shown in Proposition~2.3 in \cite{GarThi_functoriality} (here reproduced as
Proposition \ref{prop:inclusion subgp}), the map $\iota^{(d\to m)}$ induces an isometric
embedding $F^p(\Z_d)\to F^p(\Z_m)$.

We define restriction maps $\rho^{(n\to d)}_j\colon\C^n\to\C^d$ for $j\in\{0,\ldots,\frac{n}{d}-1\}$ by sending an $n$-tuple $\beta$ to the $d$-tuple $\rho^{(n\to d)}_r(\beta)$ given by
\[
\rho^{(n\to d)}(\beta)_i=\beta_{i\frac{n}{d}+j}.
\]

It follows that
\begin{align*}
\left\| f_{\alpha',t'}^{(k)}\right\|_{s',m}
&= \left\|\mu^{(kn)}(\delta) \left(\iota^{(d\to m)}\circ\rho^{(n\to d)}_0(\alpha')\right) \right\|_{F^p(\Z_m)} \\
&= \mu^{(kn)}(\delta) \left\|\rho^{(n\to d)}_0(\alpha')\right\|_{F^p(\Z_d)}.
\end{align*}

Let $b(t,s)\in\{0,\ldots,\frac{n}{d}-1\}$ be the unique number such that $\omega_m^bt$ is in the $\frac{1}{2kn}$-neighborhood of $\{s,\omega_ms,\ldots,\omega_m^{m-1}s\}$. Then
\begin{align*}
\left\| f_{\alpha,t}^{(k)}\right\|_{s,m}
&= \mu^{(kn)}(\delta) \left\|\rho^{(n\to d)}_{b(t)}(\alpha)\right\|_{F^p(\Z_d)}.
\end{align*}

Therefore
\begin{align*}
\left\| f_{\alpha,t}^{(k)} \right\|_m
&= \sup_{s\in\sigma_m} \left\| f_{\alpha,t}^{(k)} \right\|_{m,s} \\
&= \max\limits_{b=0,\ldots,\frac{m}{d}-1} \sup_{s\in\sigma_m, b(s,t)=b}
\mu^{(kn)}(\dist(s,t)) \left\|\rho^{(n\to d)}_{b}(\alpha)\right\|_{F^p(\Z_d)} \\
&= \max\limits_{b=0,\ldots,\frac{m}{d}-1} \mu^{(kn)}(\dist(\omega_m^bt,\{s,\omega_ms,\ldots,\omega_m^{m-1}s\}))
\left\|\rho^{(n\to d)}_{b}(\alpha)\right\|_{F^p(\Z_d)} \\
\end{align*}

Thus
\begin{align*}
& \ \ \ \ \ \ \ \ \ \ \ \ \ \ \ \ \ \ \ \ \ \ \ \
\lim_{k\to\infty}\left\| f_{\alpha,t}^{(k)} \right\|_{\sigma_m,m}\\
&=\max\limits_{b=0,\ldots,\frac{m}{d}-1}
\lim_{k\to\infty} \mu^{(kn)}(\dist(t\omega_m^b,\{s,\omega_ms,\ldots,\omega_m^{m-1}s\}))
\left\|\rho^{(n\to d)}_{b}(\alpha)\right\|_{F^p(\Z_d)} \\
&=\max\limits_{b=0,\ldots,\frac{m}{d}-1}
\mathbbm{1}_{\sigma_m}(\omega_m^bt)
\left\|\rho^{(n\to d)}_{b}(\alpha)\right\|_{F^p(\Z_d)}.
\end{align*}

With this computation at hand, we can `test' whether some $t\in S^1$ belongs to $\sigma_n$.
To that end, let $m\in\N$ and set $d=\gcd(m,n)$.

Assume first that $d<n$. Let $\alpha\in F^p(\Z_n)$ be as in the conclusion
of Proposition \ref{prop:RegExp}, and normalize it so that
$\max\limits_{b=0,\ldots,\frac{m}{d}-1}\left\|\rho^{(n\to d)}_{b}(\alpha)\right\|_{F^p(\Z_d)}=1$.
Then
\begin{align*}
\lim_{k\to\infty}\left\| f_{\alpha,t}^{(k)} \right\|_{\sigma_m,m}
&=\max\limits_{b=0,\ldots,\frac{m}{d}-1}
\mathbbm{1}_{\sigma_m}(\omega_m^bt)
\left\|\rho^{(n\to d)}_{b}(\alpha)\right\|_{F^p(\Z_d)}
\leq 1.
\end{align*}

On the other hand, if $d=n$ (so that $n$ divides $m$) then
\[
\lim_{k\to\infty}\left\| f_{\alpha,t}^{(k)} \right\|_{\sigma_m,m}
=\mathbbm{1}_{\sigma_m}(t) \|\alpha\|_{F^p(\Z_n)}.
\]
Thus, when $n$ divides $m$, then $\lim\limits_{k\to\infty}\left\| f_{\alpha,t}^{(k)} \right\|_{\sigma_m,m}>1$ if and only if $t\in\sigma_m$.
Again, since $n$ divides $m$, this implies $t\in\sigma_m$.

In conclusion, we have
\[
\lim\limits_{k\to\infty}\left\| f_{\alpha,t}^{(k)} \right\|_{\sigma}>1 \ \text{ if and only if }\ t\in\sigma_n.
\]
The same computations hold for $\tau$, so that we have
\[
\lim\limits_{k\to\infty}\left\| f_{\alpha,t}^{(k)} \right\|_{\tau}>1 \ \text{ if and only if } \ t\in\tau_n.
\]
By assumption, we have $\left\| f_{\beta,t}^{(k)} \right\|_{\tau}\leq\left\| f_{\beta,t}^{(k)} \right\|_{\sigma}$ for every $\beta$, $t$ and $k$.
Thus, if $t\in\tau_n$, then
\[
1<\lim_{k\to\infty}\left\| f_{\alpha,t}^{(k)} \right\|_{\tau} \leq
\lim_{k\to\infty}\left\| f_{\alpha,t}^{(k)} \right\|_{\sigma},
\]
which implies that $t\in\sigma_n$.
Hence, $\tau_n\subseteq\sigma_n$.
Since this holds for every $n$, we have shown $\tau\leq\sigma$, as desired.
\end{proof}

\begin{cor}\label{cor: isom classif of Fpsigma}
Let $\sigma$ and $\tau$ be two spectral configurations.
The following conditions are equivalent:
\begin{enumerate}
\item There is an isometric isomorphism $\varphi \colon F^p(\sigma)\to F^p(\tau)$ such that $\varphi(\iota_\sigma)=\iota_\tau$.
\item We have $\widetilde\sigma=\widetilde\tau$.
\end{enumerate}
\end{cor}

\section{Isometric isomorphisms between $L^p$-spaces.}

If $X$ is a set $E$ and $F$ are subsets of $X$, we denote by $E\triangle F$ their symmetric difference, this is,
$E\triangle F=(E\cap F^c) \cup (E^c \cap F)$.

\begin{df} Let $(X,\mu)$ and $(Y,\nu)$ be measure spaces. A \emph{measure class preserving transformation} from $X$ to $Y$
is a measurable function $T\colon X\to Y$ satisfying
$$\mu(T^{-1}(F))=0 \mbox{ if and only if } \nu(F)=0$$
for every measurable set $F\subseteq Y$.\\
\indent A measure class preserving transformation $T\colon X\to Y$ is said to be \emph{invertible}, if it is invertible as a
function from $X$ to $Y$ and its inverse is measurable.
\end{df}

Note that we do not require measure class preserving transformations to preserve the measure; just the null-sets. Also, the
inverse of a measure class preserving transformation is automatically a measure class preserving transformation.

\begin{rem} \label{rem: induced by T} Let $(X,\mathcal{A},\mu)$ and $(Y,\mathcal{B},\nu)$ be measure spaces.
Set
\[\mathcal{N}(\mu)=\{E\in \mathcal{A}\colon \mu(E)=0\}\]
and let $\mathcal{A}/\mathcal{N}(\mu)$ be the quotient of $\mathcal{A}$ by the relation $E\sim F$ if $E\triangle F$
belongs to $\mathcal{N}(\mu)$.
Define $\B/\mathcal{N}(\nu)$ analogously. If $T\colon X\to Y$ is a measure class preserving transformation, then $T$ induces a function
$$T^\ast\colon \mathcal{B}/\mathcal{N}(\nu)\to \mathcal{A}/\mathcal{N}(\mu)$$
given by $T^\ast(F+\mathcal{N}(\nu))=T^{-1}(F)+\mathcal{N}(\mu)$ for all $F$ in $\B$. Moreover, if $T$ is invertible then $T^*$
is invertible, and one has $(T^*)^{-1}=(T^{-1})^*$.\\
\indent The induced map $T^*\colon \mathcal{B}/\mathcal{N}(\nu)\to \mathcal{A}/\mathcal{N}(\mu)$ is an example of an \emph{order-continuous
Boolean homomorphism}. See Definition 313H in \cite{fremlin} and also Definition 4.9 in \cite{phillips analogs Cuntz algebras}, where they
are called $\sigma$-homomorphisms. Under fairly general assumptions on the measure spaces $(X,\mathcal{A},\mu)$ and $(Y,\B,\nu)$, every
such homomorphism $\B/\mathcal{N}(\nu)\to \mathcal{A}/\mathcal{N}(\mu)$ lifts to a measure class preserving transformation $X\to Y$.
See Theorem 343B in \cite{fremlin}.\end{rem}

We now present some examples of isometric isomorphisms between $L^p$-spaces.

\begin{egs} Let $(X,\mu)$ and $(Y,\nu)$ be measure spaces and let $p\in [1,\I)$.
\be\item For a measurable function $h\colon Y\to S^1$, the associated multiplication operator
$m_h\colon L^p(Y,\nu)\to L^p(Y,\nu)$, which is given by
$$m_h(f)(y)=h(y)f(y)$$
for all $f$ in $L^p(Y,\nu)$ and all $y$ in $Y$, is an isometric isomorphism. Indeed, its inverse is easily seen to be $m_{\overline{h}}$.
\item Let $T\colon X\to Y$ be an invertible measure class preserving transformation. Then the linear map
$u_T\colon L^p(X,\mu)\to L^p(Y,\nu)$ given by
$$u_T(f)(y)=\left( \left[\frac{d(\mu\circ T^{-1})}{d\nu}\right](y)\right)^{\frac{1}{p}} f(T^{-1}(y))$$
for all $f$ in $L^p(X,\mu)$ and all $y$ in $Y$, is an isometric isomorphism. Indeed, its inverse is $u_{T^{-1}}$.
\item If $h\colon Y\to S^1$ and $T\colon X\to Y$ are as in (1) and (2), respectively, then
\[m_h\circ u_T\colon L^p(X,\mu)\to L^p(Y,\nu)\]
is an isometric isomorphism with inverse $u_{T^{-1}}\circ m_{\overline{h}}$. \ee\end{egs}

The following result is a particular case of Lamperti's Theorem (see the Theorem in \cite{lamperti}; see also
Theorem 6.9 in \cite{phillips analogs Cuntz algebras}), and it can be regarded as a structure theorem for isometric isomorphisms
between $L^p$-spaces. It states that if $(X,\mathcal{A},\mu)$ and $(Y,\mathcal{B},\nu)$ are complete $\sigma$-finite measure spaces, then the linear operators
of the form $m_h\circ u_T$ are the only isometric isomorphisms between $L^p(X,\mu)$ and $L^p(Y,\nu)$ for $p\in [1,\I)\setminus\{2\}$.
(Although Lamperti's Theorem was originally stated and proved assuming $(X,\mathcal{A},\mu)=(Y,\B,\nu)$, this assumption was never
actually used in his proof.)\\
\indent The version we exhibit here is not the most general possible, but it is enough for our purposes.

\begin{thm} \label{thm: Lamperti}
Let $p\in [1,\I)\setminus \{2\}$. Let $(X,\mathcal{A},\mu)$ and $(Y,\mathcal{B},\nu)$ be complete
$\sigma$-finite standard Borel spaces, and let $\varphi\colon L^p(X,\mu)\to L^p(Y,\nu)$
be an isometric isomorphism. Then there exist a measurable
function $h\colon Y\to S^1$ and an invertible measure class preserving transformation $T\colon X\to Y$ such that
$$\varphi(f)(y)=h(y)\left( \left[\frac{d(\mu\circ T^{-1})}{d\nu}\right](y)\right)^{\frac{1}{p}} f(T^{-1}(y))$$
for all $f$ in $L^p(X,\mu)$ and all $y$ in $Y$. In other words, $\varphi=m_h\circ u_T$. \\
\indent Moreover, the pair $(h,T)$ is unique in the following sense: if $\widetilde{h}\colon Y\to S^1$ and $\widetilde{T}\colon X\to Y$ are,
respectively, a measurable function and an invertible measure class preserving transformation satisfying
$v=m_{\widetilde{h}} \circ u_{\widetilde{T}}$, then $h(y)=\widetilde{h}(y)$ for $\nu$-almost every $y$ in $Y$, and
\[\nu\left(T(E)\triangle \widetilde{T}(E)\right)=0\]
for every measurable subset $E\subseteq X$.
\end{thm}
\begin{proof}
We will use the language and notation from Sections 5 and 6 in \cite{phillips analogs Cuntz algebras}, where Phillips
develops the material needed to make effective use of Lamperti's Theorem (Theorem 3.1 in \cite{lamperti}). \\
\indent By Theorem 6.9 in \cite{phillips analogs Cuntz algebras}, the invertible isometry $v$ is spatial (see
Definition 6.4 in \cite{phillips analogs Cuntz algebras}). Let $(E,F,\phi,h)$ be a spatial system for $\varphi$. By Lemma
6.12 in \cite{phillips analogs Cuntz algebras} together with the fact that $\varphi$ is bijective, we deduce that
$\mu(X\setminus E)=0$ and $\nu(Y\setminus F)=0$.\\
\indent The assumptions on the measure spaces $(X,\mathcal{A},\mu)$ and $(Y,\mathcal{B},\nu)$ ensure that Theorem 343B in \cite{fremlin} applies,
so the order-continuous Boolean homomorphisms $\phi$ and $\phi^{-1}$ can be lifted to measure class preserving transformations
$f_\phi\colon X\to Y$ and $f_{\phi^{-1}}\colon Y\to X$. By Corollary 343G in \cite{fremlin}, the identities
$$f_{\phi}\circ f_{\phi^{-1}}=\mathrm{id}_Y \ \ \mbox{ and } \ \ f_{\phi^{-1}}\circ f_\phi=\mathrm{id}_X,$$
hold up to $\nu$- and $\mu$-null sets, respectively. Upon redifining them on sets of measure zero (see, for example, Theorem 344B
and Corollary 344C in \cite{fremlin}), we may take $T=f_\phi$ and $T^{-1}=f_{\phi^{-1}}$. \\
\indent We now turn to uniqueness of the pair $(h,T)$. Denote by $\mathcal{A}$ and $\B$
the domains of $\mu$ and $\nu$, respectively, and by $\mathcal{N}(\mu)$ and $\mathcal{N}(\nu)$ the subsets of $\mathcal{A}$
and $\B$ consisting of the $\mu$ and $\nu$-null sets, respectively. Uniqueness
of $h$ up to null sets, and uniqueness of $T^*\colon \mathcal{A}/\mathcal{N}(\mu)\to \mathcal{B}/\mathcal{N}(\nu)$ (see Remark
\ref{rem: induced by T} for the definition of $T^*$) were established in Lemma 6.6 in \cite{phillips analogs Cuntz algebras}. It follows from
Corollary 343G in \cite{fremlin} that $T$ is unique up to changes on sets of measure zero, which is equivalent to the
formulation in the statement.
\end{proof}

Theorem \ref{thm: Lamperti} will be crucial in our study of $L^p$-operator algebras generated by invertible isometries. Since
the invertible isometries of a given $L^p$-space are the main object of study of this work, we take a closer look at
their algebraic structure in the remainder of this section.\\

\begin{df}\label{df: notation}
Let $(X,\mu)$ be a measure space.
\be\item We denote by $L^0(X,S^1)$ the Abelian group (under pointwise multiplication) of all measurable functions $X\to S^1$, where
two such functions that agree $\mu$-almost everywhere on $X$ are considered to be the same element in $L^0(X,S^1)$.
\item We denote by $\Aut_*(X,\mu)$ the group (under composition) of all invertible measure class preserving transformations $X\to X$.
\item For $p\in [1,\I)$, we denote by $\Aut(L^p(X,\mu))$ the group (under composition) of isometric automorphisms of $L^p(X,\mu)$.
Equivalently, $\Aut(L^p(X,\mu))$ is the group of all invertible isometries of $L^p(X,\mu)$. \ee\end{df}

We point out that similar definitions and notation were introduced in \cite{pestov}.

\begin{rem} Let $(X,\mu)$ be a measure space and let $p\in [1,\I)$. Then the maps
$$m\colon L^0(X,S^1)\to \Aut(L^p(X,\mu) \ \ \mbox{ and } \ \ u\colon \Aut_*(X,\mu)\to \Aut(L^p(X,\mu)$$
given by $h\mapsto m_h$ and $T\mapsto u_T$, respectively, are injective group homomorphisms via which we may identify $L^0(X,S^1)$
and $\Aut_*(X,\mu)$ with subgroups of $\Aut(L^p(X,\mu)$.\end{rem}

\begin{prop} \label{prop: semidirect product}
Let $(X,\mathcal{A},\mu)$ be a complete $\sigma$-finite standard Borel space, and let $p\in [1,\I)$.
Then $L^0(X,S^1)$ is a normal subgroup in $\Aut(L^p(X,\mu))$,
and there is a canonical algebraic isomorphism
$$\Aut(L^p(X,\mu))\cong \Aut_*(X,\mu)\rtimes L^0(X,S^1).$$
\end{prop}
\begin{proof} Since every element in $\Aut(L^p(X,\mu))$ is of the form $m_h\circ u_T$ for some $h\in L^0(X,S^1)$ and some $T\in \Aut_*(X,\mu)$
by Theorem \ref{thm: Lamperti}, it is enough to check that conjugation by $u_T$ leaves $L^0(X,S^1)$ invariant.\\
\indent Given $h\in L^0(X,S^1)$ and $T\in \Aut_*(X,\mu)$, let $f\in L^p(X,\mu)$ and $x$ in $X$. We have
\begin{align*} \left(m_h\circ u_T\right)(f)(x) & = h(x)\left( \left[\frac{d(\mu\circ T^{-1})}{d\mu}\right](x)\right)^{\frac{1}{p}} f(T^{-1}(x))\\
 &= \left( \left[\frac{d(\mu\circ T^{-1})}{d\mu}\right](x)\right)^{\frac{1}{p}} (m_{h\circ T}f)(T^{-1}(x))\\
&= \left(u_T\circ m_{h\circ T}\right)(f)(x).
\end{align*}
We conclude that $u_{T^{-1}}\circ m_h\circ u_T=m_{h\circ T}$, which shows that $L^0(X,S^1)$ is a normal subgroup in $\Aut(L^p(X,\mu))$.\\
\indent The algebraic isomorphism $\Aut(L^p(X,\mu))\cong \Aut_*(X,\mu)\rtimes L^0(X,S^1)$ now follows from Lamperti's Theorem \ref{thm: Lamperti}.
\end{proof}

Adopt the notation of the proposition above. Endow $L^0(X,S^1)$ with the topology of convergence in measure, endow $\Aut_*(X,\mu)$ with the
weak topology, and endow $\Aut(L^p(X,\mu))$ with the strong operator topology. It is shown in \cite{pestov} that these topologies turn these
groups into Polish groups. Using this technical fact, it is shown in Corollary 3.3.46 of \cite{pestov} that the algebraic isomorphism of
the above proposition is in fact an isomorphism of topological groups.

\section{$L^p$-operator algebras generated by an invertible isometry}

We fix $p\in [1,\I)\setminus\{2\}$, and a complete $\sigma$-finite standard Borel space $(X,\mathcal{A},\mu)$.
We also fix an invertible isometry $v\colon L^p(X,\mu)\to L^p(X,\mu)$. We will introduce some
notation that will be used in most the results of this section. We will recall the standing assumptions in the
statements of the main results, but not necessarily in the intermediate lemmas and propositions.\\
\indent Using Theorem \ref{thm: Lamperti}, we choose and fix a measurable function $h\colon X\to S^1$ and an invertible measure
class preserving transformation $T\colon X\to X$ such that $v=m_h\circ u_T$. \\
\ \\
\indent Let $n$ in $\N$. Recall that a point $x$ in $X$ is said to have \emph{period $n$}, denoted $\mathcal{P}(x)=n$,
if $n$ is the least positive integer for which $T^n(x)=x$. If no such $n$ exists, we say that $x$ has \emph{infinite period},
and denote this by $\mathcal{P}(x)=\I$.\\
\indent For each $n$ in $\NI$, set
$$X_n=\{x\in X\colon \mathcal{P}(x)=n\}.$$
Then $X_n$ is measurable, and $T(X_n)\subseteq X_n$ and $T^{-1}(X_n)\subseteq X_n$ for all $n$ in $\NI$. For each $n$ in $\NI$,
denote by $h_n\colon X_n\to S^1$, by $T_n\colon X_n\to X_n$, and by $T^{-1}_n\colon X_n\to X_n$, the
restrictions of $h$, of $T$, and of $T^{-1}$ to $X_n$. Furthermore, we denote by $\mu_{|_{X_n}}$ the restriction of $\mu$
to the $\sigma$-algebra of $X_n$. \\
\indent Set $v_n=m_{h_n}\circ u_{T_n}$. Then $v_n$ is an isometric bijection
of $L^p(X_n,\mu)$. Since $X$ is the disjoint union of the sets $X_n$ for $n\in\NI$, there is an isometric isomorphism
$$L^p(X,\mu)\cong \bigoplus_{n\in \NI} L^p(X_n,\mu_{|_{X_n}})$$
under which $v$ is identified with the invertible isometry
$$\bigoplus_{n\in \NI} v_n\colon \bigoplus_{n\in \NI} L^p(X_n,\mu_{|_{X_n}})\to \bigoplus_{n\in \NI} L^p(X_n,\mu_{|_{X_n}}).$$

The next lemma shows that each of the transformations $T_n$ acts as essentially a shift of
order $n$ on $X_n$, at least when $n<\I$.

\begin{lemma} \label{lma: cross section}
Let $n$ in $\N$. Then there exists a partition $\{X_{n,j}\}_{j=0}^{n-1}$ of $X_n$ consisting of measurable subsets
such that
$$T^{-1}(X_{n,j})=X_{n,j+1}$$
for all $j\in \N$, with indices taken modulo $n$.
\end{lemma}
\begin{proof}
Note that every point of $X_n$ is periodic. It follows from Theorem 1.2 in \cite{glasner weiss} that there is a Borel cross-section,
this is, a Borel set $X_{n,0}\subseteq X_n$ such that each orbit of $T$ intersects $X_{n,0}$ exactly once. The result follows by setting
$X_{n,j}=T^{-j}(X_{n,0})$ for $j=1,\ldots,n-1$.\end{proof}

The following lemma allows us to assume that $h$ is identically equal to 1 on the sets
$X_{n,j}$ with $n<\I$ and $j>0$. The idea of the proof is to ``undo'' a certain shift on the space, which is reflected on the
construction of the functions $g_n$. One cannot undo the shift completely, and there is a remainder
left over which is concentrated on the first set $X_{n,0}$.

\begin{lemma} \label{lma: trivialization of h}
There exists a measurable function $g\colon X\to S^1$ such that the function
\[g\cdot (\overline{g}\circ T^{-1})\cdot h\colon X\to S^1\] is
identically equal to 1 on $X_{n,j}$ whenever $n<\I$ and $j>0$.
\end{lemma}
\begin{proof}
Let $n<\I$. Adopt the convention that $h\circ T^0$ is the function identically equal to 1 (this unusual convention is adopted
so that the formula below comes out nicer). Using indices modulo $n$, we define $g_n\colon X_n\to S^1$ by
$$g_n=\sum_{j=0}^{n-1} \mathbbm{1}_{X_{n,j}}\cdot(h\circ T^{0})\cdot(h\circ T)\cdots (h\circ T^{j-1}).$$
We point out that the term corresponding to $j=0$ is
$$\mathbbm{1}_{X_{n,0}}\cdot(h\circ T^{0})\cdot(h\circ T)\cdots (h\circ T^{n-1}).$$
Note that $g_n$ is well defined because the sets $X_{n,j}$ are pairwise disjoint for $j=0,\ldots,n-1$. For $n=\I$, set
$g_\I=\mathbbm{1}_{X_\I}$.\\
\indent Finally, we set
$$g=\sum_{n=1}^\I \mathbbm{1}_{X_n}\cdot g_n,$$
which is well defined because the sets $X_n$ are pairwise disjoint for $n\in\NI$. It is a routine exercise to check
that $g$ has the desired properties.
\end{proof}

\begin{rem}\label{rem: got rid of h} Let $g\colon X\to S^1$ be as in Lemma \ref{lma: trivialization of h}. A straightforward
computation shows that
$$m_{g}\circ v\circ m_{\overline{g}}= m_{g\cdot (\overline{g}\circ T^{-1})\cdot h}\circ u_T.$$
In particular, $v$ is conjugate, via the invertible isometry $m_g$, to another invertible isometry whose multiplication
component is identically equal to 1 on $X_{n,j}$ whenever $n<\I$ and $j>0$. Since $v$
and $m_g\circ v\circ m_{\overline{g}}$ generate isometrically isomorphic Banach subalgebras of $\B(L^p(X,\mu))$, we may shift our
attention to the latter isometry. Upon relabeling its multiplication component, we may and will assume that $h$ itself is
identically equal to 1 on $X_{n,j}$ whenever $n<\I$ and $j>0$. \end{rem}

Our next reduction refers to the set transformation $T$: we show that we can assume that $T$ preserves the measure of the measurable subsets
of $X\setminus X_\I$. Recall that $\mathcal{A}$ denotes the domain of $\mu$.

\begin{lemma}\label{lma: equivalent measure} There is a measure $\nu$ on $(X,\mathcal{A})$ such that
\be\item For every measurable set $E\subseteq X\setminus X_\I$, we have $\nu(T(E))=\nu(E)$.
\item For every measurable set $E\subseteq X$, we have $\nu(E)=0$ if and only if $\mu(E)=0$.\ee
Moreover, $L^p(X,\mu)$ is canonically isometrically isomorphic to $L^p(X,\nu)$.
\end{lemma}
\begin{proof} For every $n$ in $\N$, we define a measure $\nu_n$ on $X_n$ by
\[ \nu_n(E)= \sum_{j=0}^{n-1}\mu(T^{-j}(E)\cap X_{n,0})\]
for every $E$ in $\mathcal{A}$. It is clear that $\nu_n(T(E))=\nu_n(E)$ for every measurable set $E$. \\
\indent Set $\nu=\sum\limits_{n\in\N}\nu_n+\mu|_{X_\I}$. It is clear that $\nu\circ T=\nu$ on $X\setminus X_\I$, so condition (1)
is satisfied. In order to check condition (2),
assume that $\nu(E)=0$ for some measurable set $E$. If $n\in\N$, then $\nu_n(E)=0$, and thus
$$\mu(T^{-j}(E)\cap X_{n,0})=\mu(T^{-j}(E\cap X_{n,j}))=0$$
for every $j=0,\ldots,n-1$. Using that $T$ preserves null-sets, we get $\mu(E\cap X_{n,j})=0$ for
every $j=0,\ldots,n-1$. Since the sets $X_{n,j}$ form a partition of $X_n$, we deduce that $\mu(E\cap X_n)=0$ for $n<\I$.
Since we also have $\mu(E\cap X_\I)=0$ and the sets $X_n$ form a partition of $X$, we conclude that $\mu(E)=0$. \\
\indent Conversely, if $E$ is measurable and $\mu(E)=0$, then $\mu(E\cap X_\I)=0$ and $\mu(T^{-j}(E))=0$ for all $j$ in $\Z$.
Thus $\nu_n(E)=0$ for all $n$ in $\N$, whence $\nu(E)=0$.  \\
\indent The last claim is a standard fact. Denote by $f$ the Radon-Nikodym derivative $f=\frac{d\mu}{d\nu}$, and
define linear maps $\varphi_p\colon L^p(X,\mu)\to L^p(X,\nu)$ and $\psi_p\colon L^p(X,\nu)\to L^p(X,\mu)$ by
$$\varphi_p(\xi)= \xi f^{\frac{1}{p}} \ \ \mbox{ and } \ \ \psi_p(\eta)= \eta f^{-\frac{1}{p}}$$
for all $\xi$ in $L^p(X,\mu)$ and all $\eta$ in $L^p(X,\nu)$. Then $\varphi_p$ and $\psi_p$ are mutual inverses. Moreover, we have
\[\|\varphi_p(\xi)\|^p_p=\int_X |\xi|^p f \ d\nu = \int_X |\xi|^p \ d\mu = \|\xi\|_p^p\]
for all $\xi$ in $L^p(X,\mu)$. We conclude that $\varphi_p$ is the desired isometric isomorphism.
\end{proof}

\begin{rem}\label{rem: T preserves measure} Adopt the notation of Lemma \ref{lma: equivalent measure}. It is immediate to check
 that if $\varphi_p\colon L^p(X,\mu)\to L^p(X,\nu)$ is the canonical isometric isomorphism, then
$$\varphi_p(v)(\eta)(x)=h(x)\eta(T^{-1}(x))$$
for all $\eta$ in $L^p(X,\nu)$ and all $x$ in $X\setminus X_\I$. (Note the absence of the correction term which is present in
the statement of Theorem \ref{thm: Lamperti}.) Since the Banach algebras generated by $v$ and $\varphi_p(v)$ are isometrically
isomorphic, we have therefore shown that we can always assume that $T$ is measure preserving on $X\setminus X_\I$.
\end{rem}

\begin{nota} If $g\colon X\to \C$ is a measurable function, we denote by $\mbox{ran}(g)$ the range of $g$, and by $\mbox{essran}(g)$
its essential range, which is defined by
$$\mbox{essran}(g)=\left\{z \in \C\colon \mu\left(g^{-1}(B_\ep(z))\right)>0 \mbox{ for all } \ep>0\right\}.$$
It is well known that the spectrum of the multiplication operator $m_g$ is $\mbox{essran}(g)$, and that
\[\mbox{essran}(g)\subseteq \overline{\mbox{ran}(g)}.\]\end{nota}

\begin{rem}\label{rem: spectra}
We recall the following fact about spectra of elements in Banach algebras. If $A$ is a unital Banach algebra, $B$ a subalgebra containing
the unit of $A$, and $a$ is an element of $B$ such that $\spec_B(a)\subseteq S^1$, then $\spec_A(a)=\spec_B(a)$. In other words, for
elements whose spectrum (with respect to a given algebra) is contained in $S^1$, their spectrum does not change when the element is
regarded as an element of a larger or smaller algebra.
\end{rem}

The following easy lemma will be used a number of times in the proof of Theorem \ref{thm: Fpvn},
so we state and prove it separately.

\begin{lemma}\label{lma: psi isometry}
Let $n$ in $\N$, let $(Y,\nu)$ be a measure space and let $S\colon Y\to Y$ be a measurable map.
Let $E$ be a measurable subset of $Y$ with $0<\nu(E)<\I$ such that $E,S^{-1}(E),\ldots,S^{n-1}(E)$ are pairwise disjoint. Define a linear
map $\psi_E\colon \ell^p_n\to L^p(X,\mu)$ by
$$\psi_E(\eta)=\frac{1}{\nu(E)^{\frac{1}{p}}}\sum_{j=0}^{n-1}\eta_j u_S^j(\mathbbm{1}_E)$$
for all $\eta$ in $\ell^p_n$. Then $\psi_E$ is an isometry.\end{lemma}
\begin{proof} Note that $\psi_E(\eta)$ is a measurable function for all $\eta$ in $\ell^p_n$. For
$\eta$ in $\ell^p_n$, we use that the sets $E,S^{-1}(E),\ldots,S^{n-1}(E)$ are pairwise disjoint at the first step to get
$$\|\psi_E(\eta)\|_p^p=\frac{1}{\mu(E)}\sum_{j=0}^{N-1}|\eta_j|^p\|u_T^j(\mathbbm{1}_E)\|_p^p=\|\eta\|_p^p,$$
so $\psi_E$ is an isometry and the result follows.\end{proof}

Denote by $F^p(v,v^{-1})$ the unital Banach subalgebra of $\B(L^p(X,\mu))$ generated by $v$ and $v^{-1}$. Then
$F^p(v,v^{-1})$ is an $L^p$-operator algebra and there is a canonical algebra homomorphism $\C[\Z]\to F^p(v,v^{-1})$ given by $x\mapsto v$.

\begin{thm}\label{thm: Fpvn}
Let $N$ in $\N$. Then $\spec(v_N)$ is a (possibly empty) closed subset of $S^1$, which is
invariant under rotation by $\omega_N$. Moreover,
the Gelfand transform defines a canonical isometric isomorphism
$$\Gamma\colon F^p(v_N,v_N^{-1})\to \left(C(\spec(v_N)),\|\cdot\|_{\spec(v_N),N}\right).$$
(See Definition \ref{df: Fpsigma} for the definition of the norm $\|\cdot\|_{\spec(v_N),N}$.)
\end{thm}
\begin{proof} The proposition is trivial if $\spec(v_N)$ is empty (which happens if and only if $\mu(X_N)=0$), so assume it is not. \\
\indent We prove the second claim first. Note that $v^N$ is multiplication by the
measurable function
$$g=h\cdot (h\circ T^{-1})\cdots (h\circ T^{-N+1})\colon X_N\to S^1.$$
By Lemma \ref{lma: trivialization of h} and Remark \ref{rem: got rid of h}, we may assume that the range and essential of $g$ agree
with the range and essential of $h|_{X_{n,0}}$, respectively. By Exercise 6, part (b) in \cite{rieffel notes}, we may assume that the range
of $h|_{X_{n,0}}$ is contained in its essential range, and hence that $\mbox{essran}(h|_{X_{n,0}})=\overline{\mbox{ran}(h|_{X_{n,0}})}$.
Let $(z_n)_{n\in\N}$ be a dense infinite sequence in $\mbox{ran}(h|_{X_{n,0}})$, and set
$$w_n=\left( \begin{array}{ccccc}
0 & 1 &  &  & \\
& 0 & \ddots &  &  \\
& & \ddots & \ddots &  \\
&  &  & 0 & 1 \\
z_n&  &  &  & 0 \end{array} \right)\in M_N.$$
Then $w=\bigoplus\limits_{n\in\N}w_n$ is an invertible isometry on $\bigoplus\limits_{n\in\N} \ell^p_N\cong \ell^p$. \\
\ \\
\indent \textbf{Claim:} Let $f$ in $\C[\Z]$. Then $\|f(v)\|=\|f(w)\|$.\\
\indent It is clear that $\|f(w)\|=\sup\limits_{n\in\N}\left\|f(w_n)\right\|.$ Let $\ep>0$. Choose $\dt>0$ such that whenever $a$ and $b$
are elements in a Banach algebra such that $\|a-b\|<\delta$, then $\|f(a)-f(b)\|<\frac{\ep}{2}$.
Choose $n$ in $\N$ and choose $\xi=(\xi_0,\ldots,\xi_{N-1})\in \ell^p_N$ with $\|\xi\|_p=1$ such that
$$\|f(w_n)\xi\|>\|f(w)\|-\frac{\ep}{2}.$$
Since $z_n$ is in the essential range of $h|_{X_{n,0}}$, we can find a measurable set $E$ in $X$ with $\mu(E)>0$ such that
$$|h(x)-z_n|<\delta$$
for all $x$ in $E$. Since $X$ is $\sigma$-finite, we may assume that $\mu(E)<\infty$. We may also assume that the sets
$E,T(E),\ldots,T^{N-1}(E)$ are pairwise disjoint. The linear map $\psi_E\colon \ell^p_N\to L^p(X_N,\mu)$ given by
$$\psi_E(\eta)=\frac{1}{\mu(E)^{\frac{1}{p}}}\sum_{j=0}^{N-1}\eta_j u_T^j(\mathbbm{1}_E)$$
for all $\eta=(\eta_0,\ldots,\eta_{N-1})\in \ell^p_N$, is an isometry by Lemma \ref{lma: psi isometry}. \\
\indent For notational convenience, we write
$$z^{(0)}_n=z_n \ \mbox{ and } \ z^{(1)}_n=\cdots=z^{(N-1)}_n=1,$$
and we take indices modulo $N$. Thus, for $\eta$ in $\ell^p_N$ we have $w_n\eta=\left(\eta_{j-1}z_n^{(j-1)}\right)_{j=0}^{N-1}$. Moreover,
\begin{align*}
\psi_E(w_n\eta)=\frac{1}{\mu(E)^{\frac{1}{p}}}\sum_{j=0}^{N-1} z_n^{(j-1)}\eta_{j-1} u_T^j(\mathbbm{1}_E)
=\frac{1}{\mu(E)^{\frac{1}{p}}}\sum_{j=0}^{N-1} \eta_{j}z_n^{(j)} u_T^{j+1}(\mathbbm{1}_E)
\end{align*}
and
\begin{align*}
v\psi_E(\eta)&=\frac{1}{\mu(E)^{\frac{1}{p}}}\sum_{j=0}^{N-1}\eta_j (m_h\circ u_T^{j+1})(\mathbbm{1}_E)
\end{align*}
for all $\eta\in\ell^p_N$. Thus,
\begin{align*}\left\|\psi_E(w_n\eta)-v\psi_E(\eta)\right\|_p^p&=
\frac{1}{\mu(E)}\sum_{j=0}^{N-1}|\eta_j|^p \left\|z_n^{(j)} u_T^{j+1}(\mathbbm{1}_E)-(m_h\circ u_T^{j+1})(\mathbbm{1}_E)\right\|^p_p \\
&\leq \frac{1}{\mu(E)}\sum_{j=0}^{N-1}|\eta_j|^p \sup_{x\in T^{-j-1}(E)}\left|z_n^{(j)}-h(x)\right|^p\left\|u_T^{j+1}(\mathbbm{1}_E)\right\|^p_p\\
&<\|\eta\|_p^p \ \delta^p,
\end{align*}
for all $\eta\in\ell^p_N$, which shows that $\|\psi_E\circ w_n-v\circ\psi_E\|<\delta$. By the choice of $\delta$, we deduce that
$$\left\|\psi_E\circ f(w_n)-f(v)\circ\psi_E\right\|<\frac{\ep}{2}.$$
Using that $\psi_E$ is an isometry at the third step, we get
$$\|f(v)\|\geq \|f(v)\psi_E(\xi)\|\geq \left\|\psi_E(f(w_n)\xi)\right\|-\frac{\ep}{2}=\left\|f(w_n)\xi\right\|-\frac{\ep}{2}>\|f(w)\|-\ep.$$
Since $\ep>0$ is arbitrary, we conclude that $\|f(v)\|\geq \|f(w)\|$.\\
\ \\
\indent Let us show that $\|f(w)\|\geq \|f(v)\|$. Given $\ep>0$, choose $g$ in $L^p(X_N,\mu)$ with $\|g\|_p=1$ such that
$$\|f(v)g\|_p>\|f(v)\|-\ep.$$
Write $X_N$ as a disjoint union $X_N=X_{N,0}\sqcup \ldots \sqcup X_{N,N-1}$, where each of the sets $X_{N,j}$ is measurable and $T^{-1}(X_{N,j})=X_{N,j+1}$,
where the subscripts are taken mod $N$. Given a measurable subset $Y$ of $X_{N,0}$ with $0<\mu(Y)<\I$, the linear map $\psi_Y\colon
\ell^p_N\to L^p(X_N,\mu)$ defined in Lemma \ref{lma: psi isometry} is isometric because the sets $Y,T(Y),\ldots, T^{N-1}(Y)$ are pairwise
disjoint.\\
\indent Set
$$\ep_0=\min\left\{\ep, \frac{\ep}{2\|f(v)g\|_p}\right\}.$$
Choose $\delta_0>0$ such that whenever $a$ and $b$ are elements in a Banach algebra such that $\|a-b\|<\delta_0$, then $\|f(a)-f(b)\|<\ep_0.$\\
\indent By simultaneously approximating the functions
$$g_{|X_{N,0}},T^{-1}(g_{|X_{N,1}}),\ldots,T^{-N+1}(g_{|X_{N,N-1}}) \ \ \mbox{ and } \ \ h_{|X_{N,0}}$$
as functions on $X_{N,0}$, by step-functions, we can find:
\bi\item A positive integer $K$ in $\N$ and disjoint, measurable sets $Y^{(k)}$ of positive finite measure with $X_{N,0}=\bigsqcup_{k=1}^K Y^{(k)}$;
\item Elements $\eta^{(k)}=(\eta_j^{(k)})_{j=0}^{N-1}$ in $\ell^p_N$;
\item Not necessarily distinct positive integers $n_1,\ldots,n_K$,
\ei
such that
\be\item With $\widetilde{h}=\mathbbm{1}_{X\setminus X_{N,0}}+ \sum\limits_{k=1}^K z_{n_k}\mathbbm{1}_{Y^{(k)}}$, we have
$$\|h-\widetilde{h}\|_\infty <\delta_0$$
\item With $\widetilde{g}=\sum\limits_{k=1}^K \psi_{Y^{(k)}}(\eta^{(k)})$, we have
$$\|g-\widetilde{g}\|_p <\frac{\ep_0}{\|f(v)\|+\ep}.$$
\ee
Set $\widetilde{v}=m_{\widetilde{h}}\circ u_T$. It follows from condition (1) above that
$$\|v-\widetilde{v}\|\leq \|h-\widetilde{h}\|_\infty\|u_T\|<\delta_0,$$
and by the choice of $\delta_0$, we conclude that $\|f(v)-f(\widetilde{v})\|<\ep_0$.\\
\indent For $\xi\in \ell^p_N$ and $k=1,\ldots,K$, we have
\begin{align*} \widetilde{v}\psi_{Y^{(k)}}(\xi) &= (m_{\widetilde{h}}\circ u_T)\left(\frac{1}{\mu(Y^{(k)})^{\frac{1}{p}}}\sum_{j=0}^{N-1}\xi_j \mathbbm{1}_{T^{-j}(Y^{(k)})} \right)\\
 &=\frac{1}{\mu(Y^{(k)})^{\frac{1}{p}}} z_{n_k}\xi_j \mathbbm{1}_{Y^{(k)}}+\sum_{j=1}^{N-2}\xi_{j+1}\mathbbm{1}_{T^{-j}(Y^{(k)})}+\xi_0\mathbbm{1}_{T^{-N+1}(Y^{(k)})}\\
&=\psi_{Y^{(k)}}\left(w_{n_k}\xi\right).
\end{align*}
It follows that
$$f(\widetilde{v})\psi_{Y^{(k)}}(\xi)=\psi_{Y^{(k)}}(f(w_{n_k})\xi)$$
for all $\xi$ in $\ell^p_N$. Thus,
\begin{align*} \|f(v)g-f(\widetilde{v})\widetilde{g}\|_p &\leq \|f(v)g-f(\widetilde{v})g\|_p + \|f(\widetilde{v})g-f(\widetilde{v})\widetilde{g}\|_p\\
 &\leq \|f(v)-f(\widetilde{v})\| \|g\|_p + \|f(\widetilde{v})\| \|g-\widetilde{g}\|_p\\
&\leq \ep_0 + (\|f(v)\|+\ep)  \frac{\ep}{2\|f(v)g\|_p} \leq \frac{\ep}{\|f(v)g\|_p}.
\end{align*}
We therefore conclude that
\begin{align*} \frac{\|f(\widetilde{v})\widetilde{g}\|_p}{\|\widetilde{g}\|_p}\geq
\frac{\|f(v)g\|_p(1+\ep)}{\|g\|_p+\ep}=\frac{\|f(v)g\|_p(1+\ep)}{1+\ep}=\frac{\|f(v)g\|_p}{\|g\|_p}\geq \|f(v)\|-\ep.\end{align*}
We have
\begin{align*} \left\| f(\widetilde{v})\widetilde{g}\right\|_p&=\left\|f(\widetilde{v})\left(\sum_{k=1}^K \psi_{Y^{(k)}}(\eta^{(k)})\right)\right\|_p\\
&=\left\|\sum_{k=1}^K \psi_{Y^{(k)}}\left(f(w_{n_k})\eta^{(k)}\right)\right\|_p\\
&=\left(\sum_{k=1}^K \left\|f(w_{n_k})\eta^{(k)}\right\|^p_p\right)^{\frac{1}{p}},\\
\end{align*}
and also
$$\|\widetilde{g}\|_p=\left\|\sum_{k=1}^K\psi_{Y^{(k)}}(\eta^{(k)})\right\|_p = \left(\sum_{k=1}^K \left\|\eta^{(k)}\right\|^p_p\right)^{\frac{1}{p}}.$$
Set
$$\widetilde{w}=\bigoplus_{k=1}^K w_{n_k}\in \B\left(\bigoplus_{k=1}^K\ell^p_N\right),$$
and $\eta=\left(\eta^{(1)},\ldots,\eta^{(K)}\right)\in \bigoplus\limits_{k=1}^K\ell^p_N$. Then $\widetilde{w}$ is an invertible isometry, and
the computations above show that
$$\|f(\widetilde{w})\eta\|_p=\|f(\widetilde{v})\widetilde{g}\|_p \ \ \mbox{ and } \ \ \|\eta\|_p=\|\widetilde{g}\|_p.$$
We clearly have
$$\|f(\widetilde{w})\|=\max\limits_{k=1,\ldots,K} \|f(w_{n_k}) \|\leq \sup_{n\in\N} \|f(w_n) \|=\|f(w)\|.$$
We conclude that
\begin{align*}
\|f(w)\|  \geq \|f(\widetilde{w})\|\geq \frac{\|f(\widetilde{w})\eta\|_p}{\|\eta\|_p}=\frac{\|f(\widetilde{v})\widetilde{g}\|_p}{\|\widetilde{g}\|_p}
\geq \|f(v)\|-\ep.\end{align*}
Since $\ep>0$ is arbitrary, this shows that $\|f(w)\|\geq \|f(v)\|$, and hence the proof of the claim is complete.\\
\ \\
\indent We will now show that $\spec(v_N)$ is invariant under translation by the $N$-th roots of unity in $S^1$. We retain the notation
of the first part of this proof. Note that since $\spec(v_N)$ is a subset of $S^1$, it can be computed in any unital Banach
algebra that contains $v_N$ by Remark \ref{rem: spectra}; in particular, it can be computed in $F^p(v_N,v_N^{-1})$. Also, the spectrum of $v_N$ in
$F^p(v_N,v_N^{-1})$ is equal to the spectrum of $w$ in $\B(\ell^p)$, since we have shown that $v_N\mapsto w$ extends to an
isometric isomorphism $F^p(v_N,v_N^{-1})\cong F^p(w,w^{-1})$. Moreover, it is clear that
$$\spec(w)=\overline{\bigcup_{n\in\N}\spec(w_n)}.$$
It is a straightforward exercise to show that $\spec(w_n)=\{\zt\in\T\colon \zt^N=z_n\}$, which is clearly invariant under
translation by the $N$-th roots of unity in $S^1$, so the claim follows.\\
\ \\
\indent We have shown in the first part that there is an isometric isomorphism
$$F^p(v_N,v_N^{-1})\cong \left(C(\spec(v_N)),\|\cdot\|_{\spec(v_N),N}\right).$$
Now, the Gelfand transform $\Gamma\colon F^p(v_N,v_N^{-1})\to C(\spec(v_N))$ maps $v$ to the canonical inclusion of $\spec(v_N)$
into $\C$, so the image of $\Gamma$ is isometrically isomorphic to $\left(C(\spec(v_N)),\|\cdot\|_{\spec(v_N),N}\right)$, as desired.\end{proof}

The situation for $v_\I$ is rather different, since the range of the Gelfand transform does not contain all continuous functions
on its spectrum (which is either $S^1$ or empty). Indeed, the Banach algebra that $v_\I$ generates together with its inverse is
isometrically isomorphic to $F^p(\Z)$ (or the zero algebra if $\mu(X_\I)=0$); see Theorem \ref{thm: FpvI} below. \\
\indent One difficulty of working with with $v_\I$ is that the analog of Lemma \ref{lma: cross section} is not in general true,
this is, $v_\I$ need not have an infinite bilateral sub-shift, as the following example shows.

\begin{eg} Let $X=S^1$ with normalized Lebesgue measure $\mu$. Given $\theta$ in $\R\setminus\Q$, consider the invertible transformation
$T_\theta \colon S^1\to S^1$ given by rotation by angle $\theta$. Then $T_\theta$ is measure preserving and every point of $S^1$
has infinite period. We claim that there is no measurable set $E$ with positive measure such that the sets $T_\theta^n(E)$ for
$n$ in $\Z$ are pairwise disjoint. If $E$ is any set such that all its images under $T_\theta$ are disjoint, we use translation invariance of
$\mu$ to get
$$1=\mu(S^1)\geq \mu\left(\bigcup_{n\in\Z}T_\theta^n(E)\right)=\sum_{n\in\Z}\mu(E).$$
It follows that $E$ must have measure zero, and the claim is proved.
\end{eg}

In order to deal with the absence of infinite sub-shifts, we will show that the set transformation $T_\I$ has what we call
``arbitrarily long strings'', which we proceed to define.

\begin{df}\label{df: long string}
Let $(Y,\nu)$ be a measure space and let $S\colon Y\to Y$ be an invertible measure class preserving transformation. Given a measurable
set $E$ in $Y$ with $\mu(E)>0$, the finite sequence $E,S^{-1}(E),\ldots,S^{-n+1}(E)$ is called a \emph{string of length $n$} for $S$ if
the sets $E,S^{-1}(E),\ldots,S^{-n+1}(E)$ are pairwise disjoint. \\
\indent The map $S$ is said to have \emph{arbitrarily long strings} if for all $n$ in $\N$ there exists a string of length $n$.\end{df}

The following lemma is not in general true without some assumptions on the $\sigma$-algebra. What is needed in our proof is that for
every point $x$ in $X$, the intersection of the measurable sets of positive measure that contain $x$ is the singleton $\{x\}$, which is
guaranteed in our case since we are working with the (completed) Borel $\sigma$-algebra on a complete metric space.

\begin{lemma}\label{lma: long string}
Let $(Y,\nu)$ be a $\sigma$-finite measure space such that for every $y$ in $Y$, the intersection of the measurable sets of positive measure
that contain $y$ is the singleton $\{y\}$. Let $S\colon Y\to Y$ be an invertible measure class preserving transformation such that every
point of $Y$ has infinite period. Then $S$ has arbitrarily long strings.\end{lemma}
\begin{proof} Let $(E_m)_{m\in\N}$ be a decreasing sequence of measurable sets with $\mu(E_m)>0$ for all $m$ in $\N$ and such that
$\bigcap\limits_{m\in\N}E_m=\{x\}$. Without loss of generality, we may assume that $T(x)$ does not belong to $E_m$ for all $m$ in $\N$.\\
\indent Let $n$ in $\N$. We claim that there exist $m_n$ in $\N$ and a sequence $(F^{(n)}_m)_{m\geq m_n}$ of measurable sets such that
\be\item For all $m\geq m_n$, the set $F^{(n)}_m$ is contained in $E_m$ and $x\in F^{(n)}_m$;
\item $\mu(F^{(n)}_m)>0$ for all $m\geq m_n$;
\item The sets $F^{(n)}_m,T^{-1}(F^{(n)}_m),\ldots,T^{-n}(F^{(n)}_m)$ are pairwise disjoint (up to null-sets).\ee
We proceed by induction on $n$.\\
\indent Set $n=1$. If there exists $m_1$ such that $\mu(E_{m_1}\triangle T^{-1}(E_{m_1}))>0$, take $F^{(1)}_{m_1}=E_{m_1}\cap T^{-1}(E_{m_1})^c$
and $F^{(1)}_m=F^{(1)}_{m_0}\cap E_m$ for all $m\geq m_1$. It is easy to verify that the sequence $(F^{(1)}_m)_{m\geq m_1}$ satisfies
the desired properties. If no such $m_1$ exists, it follows that $\mu(E_{m}\triangle T^{-1}(E_{m}))=0$ for all $m$ in $\N$. Upon getting
rid of null-sets, we may assume that $E_m=T^{-1}(E_m)$ for all $m$ in $\N$. We have
\[\{x\}=\bigcap_{m\in\N}E_m=\bigcap_{m\in\N}T^{-1}(E_m)=T^{-1}\left(\bigcap_{m\in\N}E_m\right)=\{T^{-1}(x)\},\]
which implies that $x$ is a fixed point for $T$. This is a contradiction, and the case $n=1$ is proved.\\
\indent Let $n\geq 2$, and let $m_{n-1}$ and $(F^{(n-1)}_m)_{m\geq m_{n-1}}$ be as in the inductive hypothesis for $n-1$. Suppose that there exists
$m_n$ such that the sets
$$F^{(n-1)}_{m_n},T^{-1}(F^{(n-1)}_{m_n}),\ldots, T^{-n}(F^{(n-1)}_{m_n})$$
are not disjoint up to null-sets. Let $j\in \{0,\ldots,n-1\}$ such that
$$\mu(T^{-j}(F^{(n-1)}_{m_n})\triangle T^{-n}(F^{(n-1)}_{m_n}))>0.$$
Since $T$ preserves null-sets, it follows that $\mu(T^{-(n-j)}(F^{(n-1)}_{m_n})\triangle F^{(n-1)}_{m_n})>0$. By assumption, the first
$n-1$ translates of
$F^{(n-1)}_{m_n}$ are pairwise disjoint, so we must have $j=n$ and hence $\mu(T^{-n}(F^{(n-1)}_{m_n})\triangle F^{(n-1)}_{m_n})>0$. Using an argument
similar to the one used in the case $n=1$, one shows that the sequence given by
$$F^{(n)}_{m_n}=F^{(n-1)}_{m_n}\cap T^{-1}(F^{(n-1)}_{m_n})^c \ \ \mbox{ and } \ \ F^{(n)}_m=F^{(n)}_{m_n}\cap F^{(n-1)}_m$$
for all $m\geq m_n$, satisfies the desired properties.\\
\indent If no such $m_n$ exists, it follows that $\mu(F^{(n-1)}_m\cap T^{-n}(F^{(n-1)}_m))=0$ for all $m\geq m_{n-1}$. Again, upon getting
rid of null-sets, we may assume that $F^{(n-1)}_m=T^{-n}(F^{(n-1)}_m)$ for all $m\geq m_{n-1}$. We have
\[\{x\}=\bigcap_{m\in\N}F^{(n-1)}_m=\bigcap_{m\in\N}T^{-n}(F^{(n-1)}_m)=T^{-n}\left(\bigcap_{m\in\N}F^{(n-1)}_m\right)=\{T^{-n}(x)\},\]
and thus $T^{-n}(x)=x$. This is again a contradiction, which shows that such $m_n$ must exist. This proves the claim, and the proof is finished.
\end{proof}

To see that Lemma \ref{lma: long string} is not true in full generality, consider $Y=\Z$ endowed
with the $\sigma$-algebra $\{\emptyset,\Z\}$ and measure $\mu(\Z)=1$. Let $S\colon \Z\to\Z$ be the bilateral shift.
Then every point of $Y$ has infinite period, but there are no strings of any positive length, let alone of arbitrarily long length.

\begin{cor} The measure class preserving transformation $T\colon X\to X$ has arbitrarily long strings if and only if either
$\mu(X_\I)>0$ or $\mu(X_n)>0$ for infinitely many values of $n$ in $\N$.\end{cor}

We are now ready to prove that $F^p(v_\I,v_\I^{-1})$ is isometrically isomorphic to $F^p(\Z)$. We prove the result in greater
generality because the proof is essentially the same and the extra flexibility will be needed later.

\begin{thm}\label{thm: FpvI}
Let $p\in [1,\I)$, let $(Y,\nu)$ be a $\sigma$-finite measure space with $\nu(Y)>0$,
and let $S\colon Y\to Y$ be an invertible measure class
preserving transformation with arbitrarily long strings. Let $h\colon Y\to S^1$ be a measurable function and let
$$w=m_h\circ u_S \colon L^p(Y,\nu)\to L^p(Y,\nu)$$
be the resulting invertible isometry. Then $\spec(w)=S^1$ and there is a canonical isometric isomorphism
$$F^p(w,w^{-1})\cong F^p(\Z)$$
determined by sending $w$ to the canonical generator of $F^p(\Z)$.
\end{thm}
\begin{proof} We prove the second claim first. Denote by $\lambda_p\colon \C[\Z]\to \B(\ell^p)$ the left regular representation of $\Z$.
It is enough to show that for every $f$ in $\C[\Z]$, one has
$$\|f(w)\|_{\B(L^p(Y,\nu))}=\|\lambda_p(f)\|_{F^p(\Z)}.$$
Recall that the norm on $F^p(\Z)$ is universal with respect to representations of $\Z$ of $L^p$-spaces. Since $w$ is an invertible
isometry, it induces a representation of $\Z$ on $L^p(Y,\nu)$, and universality of the norm $\|\cdot\|_{F^p(\Z)}$ implies that
$\|\lambda_p(f)\|\geq \|f(w)\|$.\\
\indent We proceed to show the opposite inequality. Let $\ep>0$ and choose an element $\xi=(\xi_k)_{k\in\Z}$ in $\ell^p$ of finite
support with $\|\xi\|_p^p=1$, and such that
$$\|\lambda_p(f)\xi\|_p>\|\lambda_p(f)\|-\ep.$$
Choose $K$ in $\N$ such that $\xi_k=0$ whenever $|k|>K$. Find a positive integer $N$ in $\N$ and complex coefficients $a_n$ with
$-N\leq n\leq N$ such that $f(x,x^{-1})=\sum\limits_{n=-N}^Na_nx^n$.
By assumption, there exists a measurable subset $E\subseteq Y$ with
$\nu(E)>0$ such that the sets $S^{-N-K}(E),\ldots,S^{N+K}(E)$ are pairwise disjoint. Since $Y$ is $\sigma$-finite, we may assume
that $\nu(E)<\infty$.\\
\indent Define a function $g\colon Y\to \C$ by
$$g=\sum_{k=-K}^K \xi_k w^k(\mathbbm{1}_E).$$
Clearly $g$ is measurable. Using that the translates of $E$ are pairwise disjoint at the first step, we compute
$$\|g\|^p_p=\sum_{k=-K}^K |\xi_k|^p \|w^k(\mathbbm{1}_E)\|_p^p=\nu(E)\|\xi\|^p_p=\nu(E)<\infty,$$
so $g\in L^p(Y,\nu)$ and $\|g\|_p=\nu(E)^{\frac{1}{p}}$.\\
\indent We have
\begin{align*} f(w)g &= \sum_{n=-N}^N a_nw^n(g)\\
&=\sum_{n=-N}^N a_nw^n\left(\sum_{k=-K}^K\xi_kw^k(\mathbbm{1}_E)\right)\\
&=\sum_{n=-N}^N\sum_{k=-K}^K a_n\xi_kw^{n+k}(\mathbbm{1}_E)\\
&=\sum_{j=-N-K}^{N+K}\left[\lambda_p(f)\xi\right]_j w^{j}(\mathbbm{1}_E).
\end{align*}
We use again that the translates of $E$ are pairwise disjoint at the first step to get
\begin{align*} \|f(w)g\|^p_p &= \sum_{j=-N-K}^{N+K}\left|\left[\lambda_p(f)\xi\right]_j\right|^p \left\|w^{j}(\mathbbm{1}_E)\right\|
=\left\|\lambda_p(f)\xi\right\|^p_p\nu(E).
\end{align*}
We conclude that
\begin{align*} \|f(w)g\|_{p} &= \left\|\lambda_p(f)\xi\right\|_p\nu(E)^{\frac{1}{p}}\\
&= \left\|\lambda_p(f)\xi\right\|_p\|g\|_p\\
&>\left(\left\|\lambda_p(f)\right\|-\ep\right)\|g\|_p.
\end{align*}

The estimate above shows that $\|f(w)\|_{L^p(Y,\nu)}> \|\lambda_p(f)\|_{F^p(\Z)}-\ep$, and since $\ep>0$ is arbitrary, we conclude
that $\|f(w)\|\geq \|\lambda_p(f)\|_{F^p(\Z)}$, as desired. \\
\indent We now claim that $\spec(w)=S^1$. We first observe that since $\spec(w)$ is a subset of $S^1$, it can be computed in any
unital algebra that contains $w$ by Remark \ref{rem: spectra}. In particular, we can compute the spectrum in $F^p(w,w^{-1})$. \\
\indent We have shown that there is a canonical isometric isomorphism $\varphi\colon F^p(w,w^{-1})\to F^p(\Z)$ that maps $w$ to the
bilateral shift $s$ in $\B(\ell^p)$. We deduce that
$$\spec_{F^p(w,w^{-1})}(w)=\spec_{F^p(\Z)}(\varphi(w))=\spec(s)=S^1,$$
and the proof is complete.
\end{proof}

\begin{cor} Assume that $\mu(X_\I)>0$. Then $\spec(v_\I)=S^1$ and there is a canonical isometric isomorphism
$F^p(v_\I,v_\I^{-1})\cong F^p(\Z)$ determined by sending $v_\I$ to the canonical generator of $F^p(\Z)$.\end{cor}
\begin{proof} Since $\mathcal{A}$ contains the Borel $\sigma$-algebra on the metric space $X$,
Lemma \ref{lma: long string} applies and the result follows from Theorem \ref{thm: FpvI}.
\end{proof}

It is an immediate consequence of Theorem \ref{thm: Fpvn} and Theorem \ref{thm: FpvI} that the sequence $(\spec(v_n))_{n\in\NI}$ is a
spectral configuration, in the sense of Definition \ref{df: spectral conf}. We have been working with complete $\sigma$-finite standard Borel spaces in
order to use Theorem \ref{thm: Lamperti}, as well as to prove Lemma \ref{lma: long string}. Our next lemma is the first step
towards showing that one can get around this assumption in the general (separable) case.

\begin{lemma} \label{lma: well defined}
Let $p\in [1,\I)$, let $(X,\mathcal{A},\mu)$ and $(Y,\mathcal{B},\nu)$ be two complete $\sigma$-finite standard Borel spaces such that $L^p(X,\mu)$ and $L^p(Y,\nu)$
are isometrically isomorphic, and let $\varphi\colon L^p(X,\mu)\to L^p(Y,\nu)$ be any isometric isomorphism.
Let $v\colon L^p(X,\mu)\to L^p(X,\mu)$ be an invertible isometry and set $w=\varphi^{-1}\circ v \circ\varphi$, which is an
invertible isometry of $L^p(Y,\nu)$. If $X_n$ and $Y_n$, for $n$ in $\NI$, are defined as in the beginning of this section with respect to $v$ and $w$, respectively, then
$\varphi$ restricts to an isometric isomorphism $L^p(X_n,\mu|_{X_n})\to L^p(Y_n,\nu|_{Y_n})$.\end{lemma}
\begin{proof} Write $v=m_{h_v}\circ u_{T_v}$ and $w=m_{h_w}\circ u_{T_w}$ as in Theorem \ref{thm: Lamperti}. For $n$ in $\NI$, the
set $X_n$ is the set of points of $X$ of period $n$, and similarly with $Y_n$. Recall that there are canonical isometric isomorphisms
$$L^p(X,\mu)\cong \bigoplus_{n\in\NI}L^p(X_n,\mu|_{X_n}) \ \ \mbox{ and } \ \ L^p(Y,\nu)\cong \bigoplus_{n\in\NI}L^p(Y_n,\nu|_{Y_n}).$$
\indent Since $\varphi$ is an isometric isomorphism, Lamperti's theorem also applies to it, so there are a measurable function $g\colon
Y\to S^1$ and an invertible measure class preserving transformation $S\colon X\to Y$ such that $\varphi=m_g\circ u_S$. It is therefore
enough to show that $S(X_n)=Y_n$ for all $n$ in $\NI$.\\
\indent It is an easy exercise to check that the identity $w=\varphi^{-1}\circ v \circ\varphi$ implies $T_w=S^{-1}\circ T_v\circ S$. The
period of a point $x$ in $X$ (with respect to $T_v$) equals the period of $S(x)$ (with respect to $T_w$), and the result follows.\end{proof}

Recall that if $(X,\mu)$ is a measure space for which $L^p(X,\mu)$ is separable, then there is a complete $\sigma$-finite standard Borel space $(Y,\nu)$ for
which $L^p(X,\mu)$ and $L^p(Y,\nu)$ are isometrically isomorphic.

\begin{df} \label{df: sigma v}
Let $p\in [1,\I)$, let $(X,\mathcal{A},\mu)$ be a measure space for which $L^p(X,\mu)$ is separable, and let $v\colon L^p(X,\mu)\to L^p(X,\mu)$
be an invertible isometry. Let $(Y,\mathcal{B},\nu)$ be a complete $\sigma$-finite standard Borel space and let $\psi\colon L^p(X,\mu)\to L^p(Y,\nu)$ be an isometric
isomorphism. Set $w=\psi^{-1}\circ v \circ\psi$, which is an invertible isometry of $L^p(Y,\nu)$. Let $\{Y_n\}_{n\in\NI}$ be the partition
of $Y$ into measurable subsets as described in the beginning of this section, and note that $w$ restricts to an invertible isometry $w_n$
of $L^p(Y_n,\nu|_{Y_n})$ for all $n$ in $\NI$. By the comments before \ref{lma: well defined}, the sequence $(\spec(w_n))_{n\in\NI}$ is
a spectral configuration. \\
\indent We define the \emph{spectral configuration} assosited with $v$, denoted $\sigma(v)$, by
\[\sigma(v)=(\spec(w_n))_{n\in\NI}.\]\end{df}

We must argue why $\sigma(v)$, as defined above, is independent of the choice of the complete $\sigma$-finite standard Borel space $(Y,\nu)$ and the isometric
isomorphism, but this follows immediately from Lemma \ref{lma: well defined}.\\
\ \\
\indent The following is the main result of this section, and it asserts that every $L^p$-operator algebra generated by an invertible
isometry together with its inverse is as in Proposition \ref{thm: Fpsigma}. The proof will follow rather easily from the results we have
already obtained.

\begin{thm} \label{thm: description}
Let $p\in [1,\I)$, let $(X,\mu)$ be a measure space for which $L^p(X,\mu)$ is separable, and let $v\colon L^p(X,\mu)\to L^p(X,\mu)$
be an invertible isometry. Let $\sigma(v)$ be the spectral configuration associated to $v$ as in Definition \ref{df: sigma v}.
Then the Gelfand transform defines an isometric isomorphism
$$\Gamma\colon F^p(v,v^{-1})\to F^p(\sigma(v)).$$
In particular, $F^p(v,v^{-1})$ can be represented on $\ell^p$. \end{thm}
\begin{proof} Denote by $\iota\colon \spec(v) \to \C$ the canonical inclusion $\spec(v)\hookrightarrow \C$.
For $n$ in $\NI$, denote by
$\iota_n\colon \spec(v_n)\to \C$ the restriction of $\iota$ to $\spec(v_n)$, which is the canonical inclusion
$\spec(v_n)\hookrightarrow \C$. Note that $F^p(\sigma(v))$ is generated by $\iota$ and $\iota^{-1}$, and that $\Gamma(v)=\iota$.\\
\indent Let $f\in \C[\Z]$. We claim that $\|f(v)\|=\|f(\iota)\|$.
Once we have proved this, the result will follow immediately. \\
\indent Using Theorem \ref{thm: Fpvn} and Theorem \ref{thm: FpvI} at the second step, and the definition of the norm $\|\cdot\|_{\sigma(v)}$ (Definition \ref{df: Fpsigma}) at the third step, we have
\begin{align*} \|f(v)\|&=\sup_{n\in\NI}\|f(v_n)\|=\sup_{n\in\NI}\|f(\iota_n)\|_{\spec(v_n),n}=\|f(\iota)\|_{\sigma(v)},
\end{align*}
and the claim is proved.\\
\indent The last assertion in the statement follows from part (1) of Theorem \ref{thm: Fpsigma}.\end{proof}

In particular, we have shown that $L^p$-operator algebras generated by an invertible isometry and its inverse can always be
isometrically represented on $\ell^p$. We do not know whether this is special to this class of $L^p$-operator algebras,
and we believe it is possible that under relatively mild assumptions, any separable $L^p$-operator algebra can be isometrically
represented on $\ell^p$. We formally raise this a problem.

\begin{pbm} Let $p\in [1,\I)$. Find sufficient conditions for a separable $L^p$-operator algebra to be isometrically represented on $\ell^p$. \end{pbm}

It is well known that any separable $L^2$-operator algebra can be isometrically represented on $\ell^2$. \\
\ \\
\indent We combine Theorem \ref{thm: description} with Theorem \ref{thm: Fpsigma} to get an explicit description of $F^p(v,v^{-1})$
for an invertible isometry of a not necessarily separable $L^p$-space.

\begin{cor} Let $p\in [1,\I)\setminus\{2\}$, let $(X,\mu)$ be a measure space, and let
$v\colon L^p(X,\mu)\to L^p(X,\mu)$
be an invertible isometry. Then one, and only one, of the following holds:
\be\item There exist a positive integer $N\in \N$ and a (finite) spectral configuration $\sigma=(\sigma_n)_{n=1}^N$
with $\overline{\sigma}=\spec(v)$, and a canonical isometric isomorphism
$$F^p(v,v^{-1})\cong F^p(\sigma)\cong \left(C(\spec(v)),\max\limits_{n=1,\ldots,N}\|\cdot\|_{\sigma_n,n}\right).$$
In this case, there is a Banach algebra isomorphism
\[F^p(v,v^{-1})\cong \left(C(\spec(v)),\|\cdot\|_\infty\right),\]
but this isomorphism cannot in general be chosen to be isometric unless $v$ is a multiplication operator.
\item There is a canonical isometric isomorphism
\[F^p(v,v^{-1})\cong F^p(\Z).\]
\ee
\end{cor}
It is obvious that the situations described in (1) and (2) cannot both be true.
\begin{proof}
It is clear that $F^p(v,v^{-1})$ is separable as a Banach algebra. By Proposition 1.25 in \cite{phillips crossed products},
there are a measure space $(Y,\nu)$ for which $L^p(Y,\nu)$ is separable and an isometric representation $\rho\colon F^p(v,v^{-1})\to
\B(L^p(Y,\nu))$. The result now follows from Theorem \ref{thm: description}, which assumes that the isometry acts on a separable $L^p$
space, together with Theorem \ref{thm: Fpsigma}.\end{proof}

We make a few comments about what happens when $p=2$. Invertible isometries on Hilbert spaces are automatically unitary, so
one always has $C^*(v,v^{-1})=C^*(v)\cong C(\spec(v))$ isometrically, although all known proofs of this fact use completely
different methods than the ones that are used in this paper.
\newline

Recall that an algebra is said to be \emph{semisimple} if the intersection of all its maximal left (or right) ideals
is trivial. For commutative Banach algebras, this is equivalent to the Gelfand transform being injective. It is
well-known that all $C^*$-algebras are semisimple.

\begin{cor} \label{cor: semisimple, GTonto}
Let $p\in [1,\I)$, let $(X,\mu)$ be a measure space, and let
$v\colon L^p(X,\mu)\to L^p(X,\mu)$
be an invertible isometry. Then:
\be\item $F^p(v,v^{-1})$ is semisimple.
\item Except in the case when $F^p(v,v^{-1})\cong F^p(\Z)$ and $p\neq 2$, the Gelfand transform
$\Gamma\colon F^p(v,v^{-1})\to C(\spec(v))$ is an isomorphism, although not necessarily isometric.
In particular, if $\spec(v)\neq S^1$, then $\Gamma$ is an isomorphism.
\ee\end{cor}

The conclusions in Corollary \ref{cor: semisimple, GTonto} are somewhat surprising. Indeed, as we will see
in the following section, part (1) fails if $v$ is an invertible isometry of a \emph{subspace} of an $L^p$-space.
It will be shown in \cite{} that part (2) also fails for isometries of subspaces of $L^p$-spaces.
\newline

It is a standard fact that Banach algebras are closed under holomorphic functional calculus, and that $C^*$-algebras are closed under
continuous functional calculus. Using the description of the Banach algebra $F^p(v,v^{-1})$ of the corollary above, we can conclude that
algebras of this form are closed by functional calculus of a fairly big class of functions, which in some cases includes all continuous
functions on the spectrum of $v$.

\begin{cor}\label{rem: functional calculus}
Let $p\in [1,\I)\setminus\{2\}$, let $(X,\mu)$ be a measure space, and let
$v\colon L^p(X,\mu)\to L^p(X,\mu)$
be an invertible isometry. Then $F^p(v,v^{-1})$ is closed under functional calculus for continuous functions on $\spec(v)$ of
bounded variation. Moreover, if $F^p(v,v^{-1})$ is not isomorphic to $F^p(\Z)$, then it is closed under continuous functional calculus.
\end{cor}

In the context of the corollary above, if $p=2$ then $F^p(v,v^{-1})$ is always isometrically isomorphic to $C(\spec(v))$, and hence
it is closed under continuous functional calculus.

We conclude this work by describing all contractive homomorphisms between algebras of the form $F^p(v,v^{-1})$ that
respect the canonical generator.

\begin{cor}
Let $p\in [1,\I)\setminus\{2\}$, and let $v$ and $w$ be two invertible isometries on $L^p$-spaces.
The following are equivalent:
\begin{enumerate}
\item The linear map $\varphi_0\colon \C[v,v^{-1}]\to \C[w,w^{-1}]$ determined by $v\mapsto w$, extends to a contractive homomorphism
$$\varphi\colon F^p(v,v^{-1})\to F^p(w,w^{-1}).$$
\item We have $\spec(v)\subseteq \spec(w)$, and for every function $g$ in $F^p(\sigma(w))\subseteq C(\spec(w))$, the
restriction $g|_{\spec(v)}$ belongs to $F^p(\sigma(v))$ and
$$\left\|g|_{\spec(v)}\right\|_{\sigma(v)} \leq\|g\|_{\sigma(w)}.$$
\item We have $\widetilde{\sigma}(v)\subseteq \widetilde{\sigma}(w)$.
\end{enumerate}
\end{cor}
\begin{proof} This is an immediate consequence of Theorem \ref{thm: description} and Corollary \ref{cor: isom classif of Fpsigma}.\end{proof}

\section{Application: quotients of Banach algebras acting on $L^1$-spaces}

In this section, we use our description of Banach algebras generated by invertible isometries of $L^p$-spaces
to answer the case $p=1$ of a question of Le Merdy (Problem~3.8 in~\cite{LeMerdy}).
In the theorem below, we show that the quotient of a Banach algebra that acts on an $L^1$-space, cannot in general
be represented on \emph{any} $L^p$-space for $p\in [1,\I)$.
In \cite{GarThi_QuotLpOpAlgs}, we give a negative answer to the remaining cases of Le Merdy's question, again
using the results of the present work.

We begin with some preparatory notions.
Let $A$ be a commutative Banach algebra, and denote by $\Gamma\colon A\to C_0(\Max(A))$ its Gelfand transform.
Given a closed subset $E\subseteq \Max(A)$, denote by $k(E)$ the ideal
\[k(E)=\{a\in A\colon \Gamma(a)(x)=0 \ \mbox{ for all } x\in E\}\]
in $A$.
Similarly, given an ideal $I$ in $A$, set
\[h(I)=\{x\in \Max(A)\colon \Gamma(a)(x)=0 \ \mbox{ for all } a\in I\},\]
which is a closed subset of $\Max(A)$. It is clear that $h(k(E))=E$ for every closed subset $E\subseteq \Max(A)$.
It is an easy exercise to check that if $A$ is a semisimple Banach algebra, then the quotient $A/k(E)$ is
semisimple as well, essentially because $k(E)$ is the largest ideal $J$ of $A$ satisfying $h(J)=E$.

\begin{df}
A commutative, semisimple Banach algebra $A$ is said to have (or satisfy) \emph{spectral synthesis}
if for every closed subset $E\subseteq \Max(A)$, there is only one ideal $J$ in $A$ satisfying $h(J)=E$
(in which case it must be $J=k(E)$).
\end{df}

It is easy to verify that a semisimple Banach algebra $A$ has spectral synthesis if and only if every quotient
of $A$ is semisimple. (We are thankful to Chris Phillips for pointing this out to us.)

It is a classical result due to Malliavin in the late 50's (\cite{malliavin}), that for an abelian
locally compact group $G$, its group algebra $L^1(G)$ has spectral synthesis if and only if $G$ is compact.

\begin{thm}
\label{thm: L1OpAlgs}
There is a quotient of $F^1(\Z)$ that cannot be isometrically represented on \emph{any} $L^p$-space for $p\in [1,\I)$.
In particular, the class of Banach algebras that act on $L^1$-spaces is not closed under quotients.
\end{thm}
\begin{proof}
The norm of a function $f\in \ell^1(\Z)$ as an element of $F^1(\Z)=F^1_\lambda(\Z)$ is the norm it gets as a convolution
operator on $\ell^1(\Z)$. The fact that $\ell^1(\Z)$ has a contractive approximate identity is easily seen to imply
that in fact $\|f\|_{F^1(\Z)}=\|f\|_{\ell^1(\Z)}$. It follows that there is a canonical identification
$F^1(\Z)=\ell^1(\Z)$. (See Proposition~3.14 in~\cite{phillips crossed products} for a more general version of this
argument.)

Denote by $v\in\ell^1(\Z)$ the canonical invertible isometry that, together with its inverse, generates $\ell^1(\Z)$.
If $I$ is an ideal in $\ell^1(\Z)$ and $\pi\colon \ell^1(\Z)\to \ell^1(\Z)/I$ is the quotient map,
then $\ell^1(\Z)/I$ is generated by $\pi(v)$ and $\pi(v^{-1})$. These elements are invertible and have
norm one, since we have
\[\|\pi(v)\|\leq 1 \ , \|\pi(v^{-1})\|\leq 1 \ , \ \mbox{ and } \ 1\leq \|\pi(v)\|\|\pi(v^{-1})\|.\]

By Malliavin's result, $\ell^1(\Z)$ does not have spectral synthesis. Let $I$ be an ideal in $\ell^1(\Z)$ such
that $\ell^1(\Z)/I$ is not semisimple. Then this quotient cannot be represented on any $L^p$-space, for $p\in [1,\I)$,
by part~(1) of Corollary~\ref{cor: semisimple, GTonto}.
\end{proof}

We close this paper by showing that our results cannot be extended to invertible isometries of more general
Banach spaces, even subspaces of $L^p$-spaces.

\begin{rem}
We keep the notation from the proof of Theorem~\ref{thm: L1OpAlgs}.
By Corollary~1.5.2.3 in~\cite{junge}, the quotient $\ell^1(\Z)/I$ can be represented on a subspace of an
$L^1$-space. In particular, $\ell^1(\Z)/I$ is generated by an invertible isometry of a $SL^1$-space, and its
inverse. We conclude that part~(1) of Corollary~\ref{cor: semisimple, GTonto} fails if we replace $L^p$-spaces
with $SL^p$-spaces. Similarly, it follows from the results in \cite{GarThi_QuotLpOpAlgs} that part~(2) also fails
for isometries of $SL^p$-spaces, even when the Banach algebra they generate is semisimple, and even if the
invertible isometry has non-full spectrum.
\end{rem}

\end{document}